\documentclass[11pt]{amsart}
\usepackage[margin=1in]{geometry}

\usepackage{amssymb}
\usepackage{amsthm}
\usepackage{amsmath}
\usepackage{mathrsfs}
\usepackage{amsbsy}
\usepackage[all]{xy}
\usepackage{bm}
\usepackage{hyperref}
\usepackage{tikz}
\usepackage{array}
\usepackage{float}
\usepackage{enumerate}
\usepackage{xcolor}
\usepackage{hhline}
\allowdisplaybreaks
\usepackage[noadjust]{cite}

\usepackage{caption}
\usepackage{tabu}
\usepackage{diagbox}
\usepackage{enumerate, amsfonts, mathtools, thmtools,xparse,colortbl,url,hypcap,tikz,xspace,accents,enumitem,letltxmacro,comment,bbm,ytableau}
\hypersetup{colorlinks=true, citecolor=darkblue, linkcolor=darkblue,urlcolor=darkblue}
\usepackage[procnames]{listings}
\usepackage[colorinlistoftodos]{todonotes}
\usetikzlibrary{calc,through,backgrounds,shapes,matrix,math,trees}
\definecolor{darkblue}{rgb}{0.0,0,0.7} 
\definecolor{chocolate}{rgb}{0.48, 0.25, 0.0}
\definecolor{gold}{rgb}{0.83, 0.69, 0.021}

\usepackage[noabbrev,capitalise,nameinlink]{cleveref}

\newenvironment{enumerate*}
  {\begin{enumerate}[(I)]
    \setlength{\itemsep}{10pt}
    \setlength{\parskip}{0pt}}
  {\end{enumerate}}

\newtheorem{theorem}{Theorem}[section]
\newtheorem{proposition}[theorem]{Proposition}

\newtheorem{conjecture}[theorem]{Conjecture}
\newtheorem{question}[theorem]{Question}

\newtheorem{lemma}[theorem]{Lemma}

\theoremstyle{definition}
\newtheorem{definition}[theorem]{Definition}

\newtheorem{example}[theorem]{Example}

\newcommand{\lorb}{\textsc{or}}
\newcommand{\landb}{\textsc{and}}
\newcommand{\negb}{\neg}

\newcommand{\MM}{\mathrm{Ung}}

\definecolor{NormalGreen}{RGB}{0,220,0}

\newcommand{\dfn}[1]{\textcolor{blue}{\emph{#1}}}

\newcommand{\PSPACE}{\mathrm{PSPACE}}
\newcommand{\SSS}{\mathrm{SS}}
\newcommand{\odddd}{\text{strange}}
\newcommand{\oddd}{\text{weird}}
\newcommand{\Pop}{\mathsf{Pop}}
\newcommand{\pat}{\mathrm{path}}
\newcommand{\cov}{\mathrm{cov}}

\newcommand{\Tam}{\mathrm{Tam}}

\newcommand{\Atniss}{{\bf A}}
\newcommand{\Eeta}{{\bf E}}
\newcommand{\N}{\mathrm{N}}
\newcommand{\E}{\mathrm{E}}
\newcommand{\U}{\mathrm{U}}
\newcommand{\D}{\mathrm{D}}
\newcommand{\Nim}{{\sf Nim}\xspace}
\newcommand{\Chomp}{{\sf Chomp}\xspace}
\newcommand{\Nibble}{{\sf Nibble}\xspace}

\definecolor{DarkGreen}{RGB}{0,220,0}
\definecolor{SkyBlue}{RGB}{0,175,255}
\definecolor{MyPurple}{RGB}{220,0,255}
\definecolor{MyOrange}{RGB}{255,160,0}

\begin{document}

\title{The Ungar Games}
\subjclass[2010]{}

\author[Colin Defant]{Colin Defant}
\address[]{Department of Mathematics, Massachusetts Institute of Technology, Cambridge, MA 02139, USA}
\email{colindefant@gmail.com}

\author[Noah Kravitz]{Noah Kravitz}
\address[]{Department of Mathematics, Princeton University, Princeton, NJ 08540, USA}
\email{nkravitz@princeton.edu}

\author[Nathan Williams]{Nathan Williams}
\address[]{Department of Mathematical Sciences, University of Texas at Dallas,  Richardson, TX 75080, USA}
\email{nathan.williams1@utdallas.edu}

\maketitle

\begin{abstract} 
Let $L$ be a finite lattice. Inspired by Ungar's solution to the famous \emph{slopes problem}, we define an \emph{Ungar move} to be an operation that sends an element $x\in L$ to the meet of $\{x\}\cup T$, where $T$ is a subset of the set of elements covered by $x$. We introduce the following \emph{Ungar game}. Starting at the top element of $L$, two players---Atniss and Eeta---take turns making nontrivial Ungar moves; the first player who cannot do so loses the game.  Atniss plays first.  We say $L$ is an \emph{Atniss win} (respectively, \emph{Eeta win}) if Atniss (respectively, Eeta) has a winning strategy in the Ungar game on $L$. We first prove that the number of principal order ideals in the weak order on $S_n$ that are Eeta wins is $O(0.95586^nn!)$. We then consider a broad class of intervals in Young's lattice that includes all principal order ideals, and we characterize the Eeta wins in this class; we deduce precise enumerative results concerning order ideals in rectangles and type-$A$ root posets. We also characterize and enumerate principal order ideals in Tamari lattices that are Eeta wins. Finally, we conclude with some open problems and a short discussion of the computational complexity of Ungar games.
\end{abstract}

\section{Introduction}\label{sec:intro}
\subsection{Poset Games}

In Gale's game \Chomp~\cite{gale1974curious}, we begin with a rectangular chocolate bar whose northwestmost \emph{carr\'e}\footnote{It appears that there is no generally accepted English word for this concept.  Hershey's has attempted to popularize the word ``pip,'' but this has not caught on.  The word ``square'' may be the closest approximation in standard English usage.} has been removed. Two players alternately take nonempty bites, where each bite consists of choosing a \emph{carr\'e} and eating all \emph{carr\'es} that lie weakly southeast of the chosen one. The first player who is left with nothing to eat is designated the loser. See~\Cref{fig:chomp_and_nibble} for an example.  Although it is easy to see (using a strategy-stealing argument, as described in Gale's original paper) that the first player can always guarantee a win in this game, describing an explicit winning strategy is open even for $3$-row chocolate bars~\cite{Zeilberger}. 

More generally, one can play \Chomp on a chocolate bar of an arbitrary skew partition shape; at this level of generality, the first player does not always have a winning strategy. In fact, \Chomp generalizes even further to finite posets (without mention to chocolate).  In the \dfn{poset game} played on the finite poset $P$, two players start with $P$ and then alternately remove nonempty principal upward-closed sets; the first player who is unable to make a move (i.e., who is left with the empty set) loses. \Nim, another notable example of a poset game, corresponds to the case where $P$ is a disjoint union of chains.

\newcommand{\pip}[2]{
\begin{scope}[shift={(#1,#2)}]
    \draw[color=chocolate,fill=chocolate!80!white,very thick] (0,0) -- (0,1) -- (1,1) -- (1,0) -- (0,0);
    \draw[color=chocolate,fill=chocolate,very thick] (0,0) -- (.2,.2) -- (.2,.8) -- (.8,.8) -- (1,1) -- (1,0) -- (0,0);
    \draw[color=chocolate!80!white,very thick,line cap=round] (.2,.2) -- (.8,.2) -- (.8,.8);
\end{scope}
}
\usetikzlibrary{arrows, decorations.pathmorphing}
\begin{figure}[htbp]
\scalebox{0.8}{
\begin{tikzpicture}[scale=2,
       shorten > = 1pt,auto,
   node distance = 3cm,
      decoration = {snake,   
                    pre length=3pt,post length=7pt,
                    }]
    \node (name) at (3,7.5) {\scalebox{1.25}{$\Chomp$}};
    \node (32) at (4,7) {\scalebox{0.5}{\begin{tikzpicture}\draw[fill=black!20,draw=none] (0,1)--(0,2)--(1,2)--(1,1)--(0,1);\draw[dashed,thick,color=chocolate] (1,2) -- (1,1) -- (0,1);\draw[dashed,thick,color=chocolate] (0,0) -- (3,0) -- (3,2) -- (0,2) -- (0,0);;\pip{1}{1};\pip{2}{1};\pip{0}{0};\pip{1}{0};\pip{2}{0};\end{tikzpicture}}};
    \node (22) at (3,6) [rectangle,draw,color=gold,fill=gold!70] {\scalebox{0.5}{\begin{tikzpicture}\draw[fill=black!20,draw=none] (0,1)--(0,2)--(1,2)--(1,1)--(0,1);\draw[dashed,thick,color=chocolate] (1,2) -- (1,1) -- (0,1);\draw[dashed,thick,color=chocolate] (0,0) -- (3,0) -- (3,2) -- (0,2) -- (0,0);\pip{1}{1};\pip{2}{1};\pip{0}{0};\pip{1}{0};\end{tikzpicture}}};
    \node (12) at (2,5) {\scalebox{0.5}{\begin{tikzpicture}\draw[fill=black!20,draw=none] (0,1)--(0,2)--(1,2)--(1,1)--(0,1);\draw[dashed,thick,color=chocolate] (1,2) -- (1,1) -- (0,1);\draw[dashed,thick,color=chocolate] (0,0) -- (3,0) -- (3,2) -- (0,2) -- (0,0);\pip{1}{1};\pip{2}{1};\pip{0}{0};\end{tikzpicture}}};
    \node (21) at (4,5) {\scalebox{0.5}{\begin{tikzpicture}\draw[fill=black!20,draw=none] (0,1)--(0,2)--(1,2)--(1,1)--(0,1);\draw[dashed,thick,color=chocolate] (1,2) -- (1,1) -- (0,1);\draw[dashed,thick,color=chocolate] (0,0) -- (3,0) -- (3,2) -- (0,2) -- (0,0);\pip{1}{1};\pip{0}{0};\pip{1}{0};\end{tikzpicture}}};
    \node (02) at (1,4) {\scalebox{0.5}{\begin{tikzpicture}\draw[fill=black!20,draw=none] (0,1)--(0,2)--(1,2)--(1,1)--(0,1);\draw[dashed,thick,color=chocolate] (1,2) -- (1,1) -- (0,1);\draw[dashed,thick,color=chocolate] (0,0) -- (3,0) -- (3,2) -- (0,2) -- (0,0);\pip{1}{1};\pip{2}{1};\end{tikzpicture}}};
    \node (11) at (3,4) [rectangle,draw,color=gold,fill=gold!70] {\scalebox{0.5}{\begin{tikzpicture}\draw[fill=black!20,draw=none] (0,1)--(0,2)--(1,2)--(1,1)--(0,1);\draw[dashed,thick,color=chocolate] (1,2) -- (1,1) -- (0,1);\draw[dashed,thick,color=chocolate] (0,0) -- (3,0) -- (3,2) -- (0,2) -- (0,0);\pip{1}{1};\pip{0}{0};\end{tikzpicture}}};
    \node (01) at (2,3) {\scalebox{0.5}{\begin{tikzpicture}\draw[fill=black!20,draw=none] (0,1)--(0,2)--(1,2)--(1,1)--(0,1);\draw[dashed,thick,color=chocolate] (1,2) -- (1,1) -- (0,1);\draw[dashed,thick,color=chocolate] (0,0) -- (3,0) -- (3,2) -- (0,2) -- (0,0);\pip{1}{1};\end{tikzpicture}}};
    \node (10) at (4,3) {\scalebox{0.5}{\begin{tikzpicture}\draw[fill=black!20,draw=none] (0,1)--(0,2)--(1,2)--(1,1)--(0,1);\draw[dashed,thick,color=chocolate] (1,2) -- (1,1) -- (0,1);\draw[dashed,thick,color=chocolate] (0,0) -- (3,0) -- (3,2) -- (0,2) -- (0,0);\pip{0}{0};\end{tikzpicture}}};
    \node (0) at (3,2) [rectangle,draw,color=gold,fill=gold!70] {\scalebox{0.5}{\begin{tikzpicture}\draw[fill=black!20,draw=none] (0,1)--(0,2)--(1,2)--(1,1)--(0,1);\draw[dashed,thick,color=chocolate] (0,0) -- (3,0) -- (3,2) -- (0,2) -- (0,0);\draw[dashed,thick,color=chocolate] (1,2) -- (1,1) -- (0,1);\end{tikzpicture}}};
    \draw[->,very thick] (32) to (22);
    \draw[->,very thick] (32) to (21);
    \draw[->,very thick] (32) to[bend right=35] (12);
    \draw[->,very thick] (32) to[bend left=30] (10);
    \draw[->,very thick] (32) to[bend right=45] (02);
    \draw[->,very thick] (22) to (12);
    \draw[->,very thick] (22) to (21);
    \draw[->,very thick] (22) to (10);
    \draw[->,very thick] (22) to[bend right=30] (02);
    \draw[->,very thick] (12) to (02);
    \draw[->,very thick] (12) to (11);
    \draw[->,very thick] (12) to[bend right=30] (10);
    \draw[->,very thick] (21) to (11);
    \draw[->,very thick] (02) to (01);
    \draw[->,very thick] (02) to[bend right=30] (0);
    \draw[->,very thick] (21) to (10);
    \draw[->,very thick] (21) to[bend right=35] (01);
    \draw[->,very thick] (11) to (10);
    \draw[->,very thick] (11) to (01);
    \draw[->,very thick] (10) to (0);
    \draw[->,very thick] (01) to (0);
\end{tikzpicture}}\qquad\qquad
\scalebox{0.8}{\begin{tikzpicture}[scale=2,
       shorten > = 1pt,auto,
   node distance = 3cm,
      decoration = {snake,   
                    pre length=3pt,post length=7pt,
                    }]
    \node (name) at (3,7.5) {\scalebox{1.25}{$\Nibble$}};
    \node (32) at (4,7) [rectangle,draw,color=gold,fill=gold!70] {\scalebox{0.5}{\begin{tikzpicture}\draw[fill=black!20,draw=none] (0,1)--(0,2)--(1,2)--(1,1)--(0,1);\draw[dashed,thick,color=chocolate] (1,2) -- (1,1) -- (0,1);\draw[dashed,thick,color=chocolate] (0,0) -- (3,0) -- (3,2) -- (0,2) -- (0,0);;\pip{1}{1};\pip{2}{1};\pip{0}{0};\pip{1}{0};\pip{2}{0};\end{tikzpicture}}};
    \node (22) at (3,6) {\scalebox{0.5}{\begin{tikzpicture}\draw[fill=black!20,draw=none] (0,1)--(0,2)--(1,2)--(1,1)--(0,1);\draw[dashed,thick,color=chocolate] (1,2) -- (1,1) -- (0,1);\draw[dashed,thick,color=chocolate] (0,0) -- (3,0) -- (3,2) -- (0,2) -- (0,0);\pip{1}{1};\pip{2}{1};\pip{0}{0};\pip{1}{0};\end{tikzpicture}}};
    \node (12) at (2,5) {\scalebox{0.5}{\begin{tikzpicture}\draw[fill=black!20,draw=none] (0,1)--(0,2)--(1,2)--(1,1)--(0,1);\draw[dashed,thick,color=chocolate] (1,2) -- (1,1) -- (0,1);\draw[dashed,thick,color=chocolate] (0,0) -- (3,0) -- (3,2) -- (0,2) -- (0,0);\pip{1}{1};\pip{2}{1};\pip{0}{0};\end{tikzpicture}}};
    \node (21) at (4,5) [rectangle,draw,color=gold,fill=gold!70] {\scalebox{0.5}{\begin{tikzpicture}\draw[fill=black!20,draw=none] (0,1)--(0,2)--(1,2)--(1,1)--(0,1);\draw[dashed,thick,color=chocolate] (1,2) -- (1,1) -- (0,1);\draw[dashed,thick,color=chocolate] (0,0) -- (3,0) -- (3,2) -- (0,2) -- (0,0);\pip{1}{1};\pip{0}{0};\pip{1}{0};\end{tikzpicture}}};
    \node (02) at (1,4) [rectangle,draw,color=gold,fill=gold!70] {\scalebox{0.5}{\begin{tikzpicture}\draw[fill=black!20,draw=none] (0,1)--(0,2)--(1,2)--(1,1)--(0,1);\draw[dashed,thick,color=chocolate] (1,2) -- (1,1) -- (0,1);\draw[dashed,thick,color=chocolate] (0,0) -- (3,0) -- (3,2) -- (0,2) -- (0,0);\pip{1}{1};\pip{2}{1};\end{tikzpicture}}};
    \node (11) at (3,4) {\scalebox{0.5}{\begin{tikzpicture}\draw[fill=black!20,draw=none] (0,1)--(0,2)--(1,2)--(1,1)--(0,1);\draw[dashed,thick,color=chocolate] (1,2) -- (1,1) -- (0,1);\draw[dashed,thick,color=chocolate] (0,0) -- (3,0) -- (3,2) -- (0,2) -- (0,0);\pip{1}{1};\pip{0}{0};\end{tikzpicture}}};
    \node (01) at (2,3) {\scalebox{0.5}{\begin{tikzpicture}\draw[fill=black!20,draw=none] (0,1)--(0,2)--(1,2)--(1,1)--(0,1);\draw[dashed,thick,color=chocolate] (1,2) -- (1,1) -- (0,1);\draw[dashed,thick,color=chocolate] (0,0) -- (3,0) -- (3,2) -- (0,2) -- (0,0);\pip{1}{1};\end{tikzpicture}}};
    \node (10) at (4,3) {\scalebox{0.5}{\begin{tikzpicture}\draw[fill=black!20,draw=none] (0,1)--(0,2)--(1,2)--(1,1)--(0,1);\draw[dashed,thick,color=chocolate] (1,2) -- (1,1) -- (0,1);\draw[dashed,thick,color=chocolate] (0,0) -- (3,0) -- (3,2) -- (0,2) -- (0,0);\pip{0}{0};\end{tikzpicture}}};
    \node (0) at (3,2) [rectangle,draw,color=gold,fill=gold!70] {\scalebox{0.5}{\begin{tikzpicture}\draw[fill=black!20,draw=none] (0,1)--(0,2)--(1,2)--(1,1)--(0,1);\draw[dashed,thick,color=chocolate] (0,0) -- (3,0) -- (3,2) -- (0,2) -- (0,0);\draw[dashed,thick,color=chocolate] (1,2) -- (1,1) -- (0,1);\end{tikzpicture}}};
    \draw[->,very thick] (32) to (22);
    \draw[->,very thick] (22) to (12);
    \draw[->,very thick] (22) to (21);
    \draw[->,very thick] (22) to (11);
    \draw[->,very thick] (12) to (02);
    \draw[->,very thick] (12) to (11);
    \draw[->,very thick] (12) to (01);
    \draw[->,very thick] (21) to (11);
    \draw[->,very thick] (02) to (01);
    \draw[->,very thick] (11) to (10);
    \draw[->,very thick] (11) to (01);
    \draw[->,very thick] (11) to (0);
    \draw[->,very thick] (10) to (0);
    \draw[->,very thick] (01) to (0);
\end{tikzpicture}}
\caption{Allowable moves in \Chomp (left) and in \Nibble (right).  Starting positions for which the second player has a winning strategy are indicated in gold.}
\label{fig:chomp_and_nibble}
\end{figure}
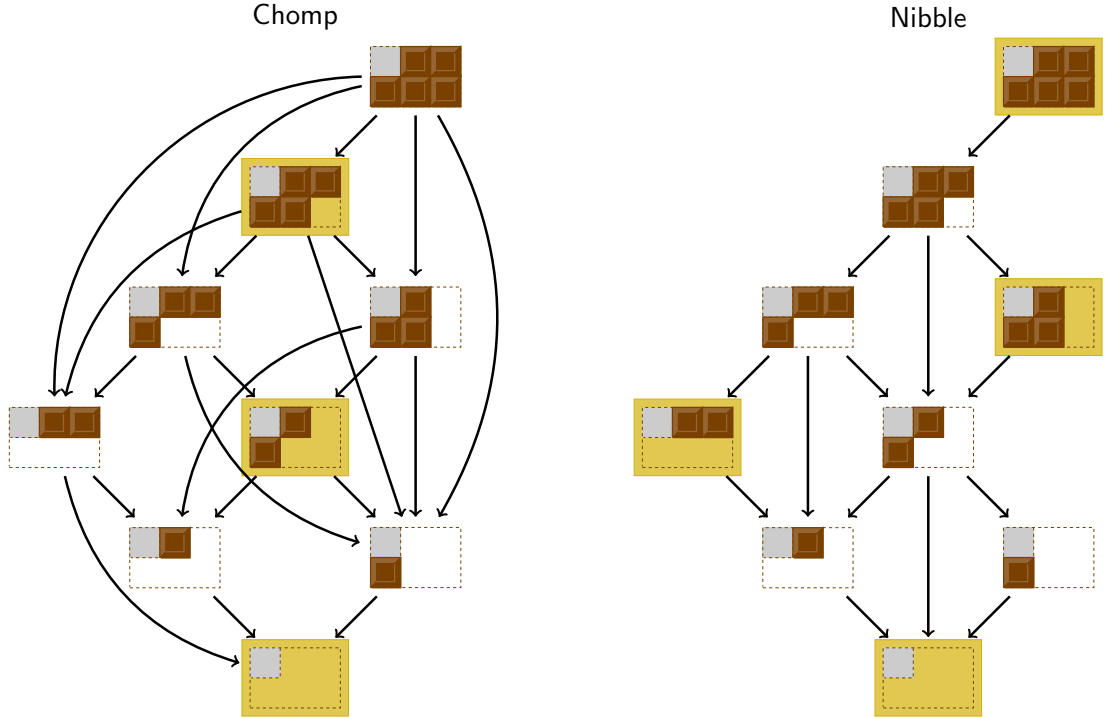

\subsection{Nibble}\label{subsec:nibble}

Consider now the following more genteel version of \Chomp, which we call \Nibble.  Instead of taking a boorishly large mouthful, a player may only politely nibble away at any number of exposed corner \emph{carr\'es} of the chocolate bar.  An example is illustrated in~\Cref{fig:chomp_and_nibble}.  A corollary of one of our main results (\Cref{thm:Young}) is a complete characterization of which player has a winning strategy when \Nibble is played on a chocolate bar in the shape of an arbitrary Young diagram.  

Just as \Chomp generalizes to arbitrary finite posets, \Nibble generalizes to arbitrary finite lattices; our next order of business is explaining this generalization.

\subsection{Ungar Moves}\label{subsec:UngarMoves}
In 1970, Scott \cite{Scott} asked for the minimum possible number of distinct slopes determined by a collection of $n\geq 4$ points in the plane that do not all lie on a single line. Ungar \cite{Ungar} solved this problem in 1982 by showing that the answer is $2\lfloor n/2\rfloor$. Building on an approach suggested by Goodman and Pollack \cite{Goodman}, Ungar considered projecting the collection of points onto a rotating line. At each  point in time, the ordering of the projected points along the line yields a permutation of the set $[n]=\{1,\ldots,n\}$. As the line rotates, the projected points sometimes swap positions in the ordering. (See \cref{fig:rotate}.) This idea allowed Ungar to work in a purely combinatorial setting in which he analyzed certain \emph{moves} that can be performed on permutations. Each such move reverses some disjoint consecutive decreasing subsequences of a permutation. For instance, we could reverse the consecutive decreasing subsequences $53$ and $641$ in the permutation $853297641$ to obtain the new permutation $835297146$. 

\begin{figure}[ht]
  \begin{center}{\includegraphics[height=10cm]{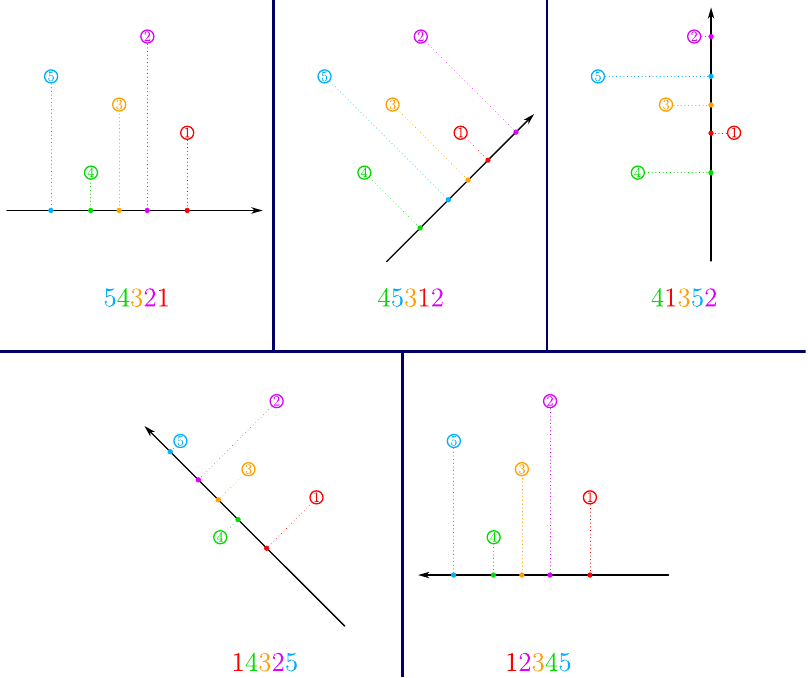}}
  \end{center}
  \caption{Five points in the plane are numbered ${\color{red}1},{\color{MyPurple}2},{\color{MyOrange}3},{\color{NormalGreen}4},{\color{SkyBlue}5}$. One can project the points onto a line and read the ordering of the projections along the line to obtain a permutation. When the line rotates, the associated permutation changes via an Ungar move.}\label{fig:rotate}
\end{figure}

Every poset $P$ in this article is assumed to have the property that $\{y\in P:y\leq x\}$ is finite for every $x\in P$. Given a poset $P$ and an element $x\in P$, we write $\cov_P(x)$ for the set of elements of $P$ that are covered by $x$. There is an equivalent way of formulating the moves that Ungar studied if we view the symmetric group $S_n$ as a lattice under the (right) weak order: a move sends a permutation $w\in S_n$ to the meet $\bigwedge\left(\{w\}\cup T\right)$, where $T\subseteq \cov_{S_n}(w)$. This observation leads to the following much more general definition from \cite{DefantLiUngarian}. 

\begin{definition}[\cite{DefantLiUngarian}]
Let $L$ be a meet-semilattice. An \dfn{Ungar move} is an operation that sends an element $x\in L$ to $\bigwedge (\{x\}\cup T)$ for some set $T\subseteq\cov_L(x)$. We say this Ungar move is \dfn{trivial} if $T=\emptyset$, and we say it is \dfn{maximal} if $T=\cov_L(x)$. 
\end{definition}

Given a meet-semilattice $L$ and an element $x\in L$, we write $\MM(x)$ for the set of elements of $L$ that can be obtained by applying an Ungar move to $x$.

Suppose $n\geq 4$, and, as before, view $S_n$ as a lattice under the weak order. Consider starting with the decreasing permutation with one-line notation $n(n-1)\cdots 1$ and applying nontrivial Ungar moves until reaching the identity permutation $12\cdots n$. Ungar proved that if the first Ungar move in this process is not maximal, then the total number of Ungar moves needed is at least $2\lfloor n/2\rfloor$~\cite{Ungar}. This allowed him to resolve Scott's original geometric problem about slopes. See \cite[Chapter~12]{PftB} for additional exposition about this result.

A meet-semilattice $L$ has an associated \dfn{pop-stack sorting operator} $\Pop_L\colon L\to L$, which acts on each element of $L$ by applying a maximal Ungar move.  The nomenclature comes from the fact that $\Pop_{S_n}$ coincides with a map that sends a permutation through a data structure called a \emph{pop-stack}---this map has been the object of considerable study in combinatorics and theoretical computer science \cite{Asinowski, Asinowski2, ClaessonPop, ClaessonPantone, Lichev, PudwellSmith}, and numerous recent articles have investigated pop-stack sorting operators on other interesting lattices \cite{ChoiSun, DefantCoxeterPop, DefantMeeting, Semidistrim, Hong, Sapounakis}. In his original paper \cite{Ungar}, Ungar also proved that the maximum number of iterations of $\Pop_{S_n}$ needed to send a permutation in $S_n$ to the identity is $n-1$.

In \cite{DefantLiUngarian}, the first author and Li studied \emph{Ungarian Markov chains}, which are random processes on lattices in which Ungar moves are applied randomly.  

\subsection{Ungar Games}

We can now describe our generalization of \Nibble.  Let $L$ be a finite lattice.  Starting at the top element $\hat 1\in L$, two players---Atniss and Eeta---take turns making nontrivial Ungar moves; the first player who cannot make a nontrivial Ungar move loses the game. We assume that Atniss goes first. Note that the game ends precisely when a player reaches the bottom element $\hat 0$ of $L$. In particular, Eeta wins if $|L|=1$. Observe that exactly one of the two players has a winning strategy in the Ungar game on $L$.  

\begin{definition}
We say a finite lattice $L$ is an \dfn{Atniss win} if Atniss has a winning strategy in the Ungar game played on $L$; otherwise, we say $L$ is an \dfn{Eeta win}. 
\end{definition}

Slightly abusing terminology, we also make the following definition when we have a fixed meet-semilattice. 

\begin{definition}
Given a meet-semilattice $L$ with a minimal element $\hat 0$, we say an element $x\in L$ is an \dfn{Atniss win} in $L$ if the interval $[\hat 0,x]$ in $L$ is an Atniss win; otherwise, we say $x$ is an \dfn{Eeta win} in $L$. Let $\Atniss(L)$ and $\Eeta(L)$ denote the set of Atniss wins in $L$ and the set of Eeta wins in $L$, respectively. 
\end{definition}

\begin{figure}[hb!]
  \begin{center}{\includegraphics[height=7.784cm]{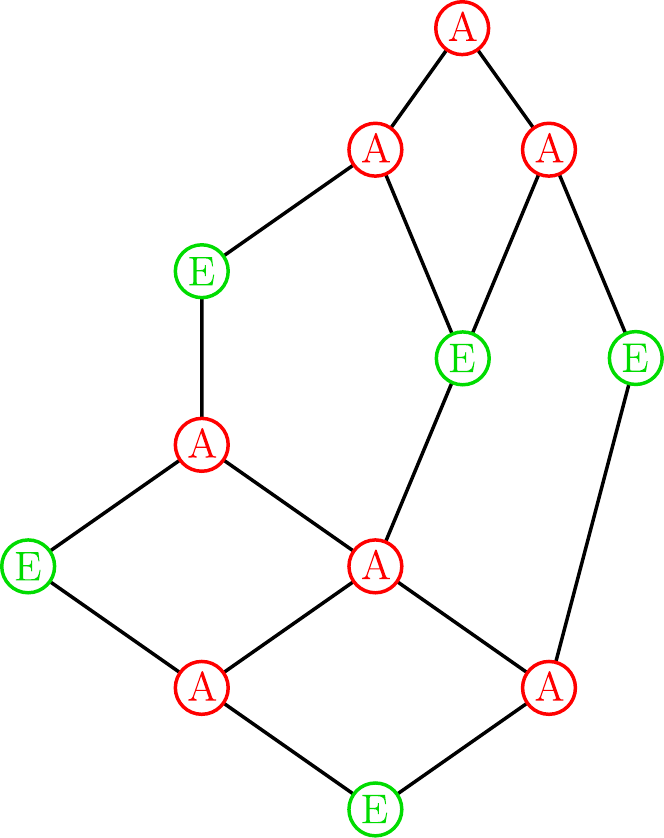}}
  \end{center}
\caption{A lattice with $7$ Atniss wins (labeled {\color{red}A}) and $5$ Eeta wins (labeled {\color{NormalGreen}E}). The entire lattice is an Atniss win.}\label{fig:arbitrary_lattice}
\end{figure}

One can determine the sets $\Atniss(L)$ and $\Eeta(L)$ recursively.  First, the bottom element $\hat 0$ is an Eeta win. In general, an element $x\in L$ is an Atniss win if there {\color{DarkGreen}e}xists an {\color{DarkGreen}E}eta win in $\MM(x)\setminus\{x\}$, while $x$ is an Eeta win if the elements of $\MM(x)\setminus\{x\}$ are {\color{red}a}ll {\color{red}A}tniss wins. See \Cref{fig:arbitrary_lattice}.

Our main focus in this article will be the characterization and the asymptotic and/or exact enumeration of Eeta (equivalently, Atniss) wins in various interesting lattices.  

\subsection{The Weak Order}
We begin by considering the weak order on $S_n$ since that is, after all, the context in which Ungar moves first arose. 

\begin{theorem}\label{thm:Weak}
We have $|\Eeta(S_n)|=O(0.95586^nn!)$. 
\end{theorem}

Although the preceding theorem is not as precise as the results that we will derive for other lattices, it still shows that asymptotically almost all elements of $S_n$ are Atniss wins. 

\subsection{Intervals in Young's Lattice}\label{subsec:Young}

We write $J(P)$ for the lattice of finite order ideals of a poset $P$, ordered by containment. \dfn{Young's lattice} is $J(\mathbb N^2)$; equivalently, it is the lattice of integer partitions ordered by containment of Young diagrams. We tacitly identify integer partitions and skew partitions\footnote{Although it is customary to write $\lambda/\mu$ for the skew shape obtained by removing a Young diagram $\mu$ from a Young diagram $\lambda$, we will break with this convention and write $\lambda\setminus\mu$ instead. This is because we view $\lambda$ and $\mu$ as posets.} with their Young diagrams (which we draw using English conventions). Suppose $\mu$ and $\lambda$ are partitions with $\mu\leq\lambda$. We can view $\lambda\setminus\mu$ as a poset whose elements are the boxes of $\lambda\setminus\mu$; the order relation is such that $\Box\leq\Box'$ if and only if $\Box$ lies weakly northwest of $\Box'$. The interval $[\mu,\lambda]$ in Young's lattice is naturally isomorphic to $J(\lambda\setminus\mu)$, the lattice of order ideals of $\lambda\setminus\mu$. The Ungar game on $J(\lambda\setminus\mu)$ is equivalent to the game \Nibble played on a chocolate bar of shape $\lambda\setminus\mu$. 

A \dfn{lattice path} is a finite path that starts at a point in $\mathbb Z^2$ and uses unit north (i.e., $(0,1)$) steps and unit east (i.e., $(1,0)$) steps. We denote north steps by $\text{N}$ and east steps by $\text{E}$, and we identify lattice paths with finite words over the alphabet $\{\text{N},\text{E}\}$. A \dfn{block} of a lattice path is a maximal consecutive string of steps that have the same direction. For example, the blocks of the lattice path $\E\E\N\E\N\N\N\N$ are $\E\E$, $\N$, $\E$, and $\N\N\N\N$ (in that order).

Associated to a partition $\lambda$ is the lattice path $\pat(\lambda)$ obtained by traversing the southeast boundary of $\lambda$. More precisely, if $\lambda=(\lambda_1,\ldots,\lambda_k)$, where $\lambda_1\geq\cdots\geq\lambda_k\geq 1$, then \[\pat(\lambda)=\E^{\lambda_k}\N\E^{\lambda_{k-1}-\lambda_k}\N\cdots\E^{\lambda_1-\lambda_2}\N.\] 
The $n$-th \dfn{staircase} (for $n \geq 0$) is the partition $\delta_n=(n,n-1,\ldots,2,1)$; its associated lattice path is $\pat(\delta_n)=(\E\N)^n$.  

The following theorem treats a very large class of intervals in Young's lattice and characterizes which of them are Eeta wins. In particular, the case $\mu=\emptyset$ (and $n=0$) completely characterizes which elements of Young's lattice are Eeta wins. 

\begin{theorem}\label{thm:Young}
Consider an interval $[\mu,\lambda]$ in Young's lattice. Let $n$ be the smallest integer such that $\mu\leq\delta_n$. If $\delta_{n+1}\leq\lambda$, then the interval $[\mu,\lambda]$ is an Eeta win if and only if $\pat(\lambda)$ does not contain an odd-length block of east steps immediately followed by an odd-length block of north steps. 
\end{theorem}

\begin{example}
Let $\mu$ be the partition $(3,1)$. Then $\mu\leq\delta_3$, but $\mu\not\leq\delta_2$. Therefore, we can apply \cref{thm:Young} whenever $\delta_4\leq\lambda$. 

If $\lambda=(5,4,2,2)$, then the Young diagram of $\lambda\setminus\mu$ is \[\begin{array}{l}\includegraphics[height=1.459cm]{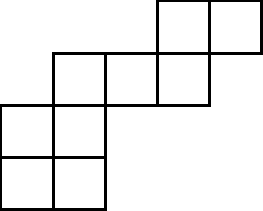}\end{array}.\] In this case, \cref{thm:Young} tells us that $[\mu,\lambda]$ is an Atniss win because $\pat(\lambda)=\E\E\N\N\E\E\N\E\N$ contains a block consisting of a single east step immediately followed by a block consisting of a single north step. 

On the other hand, if $\lambda=(6,4,2,2)$, then the Young diagram of $\lambda\setminus\mu$ is \[\begin{array}{l}\includegraphics[height=1.459cm]{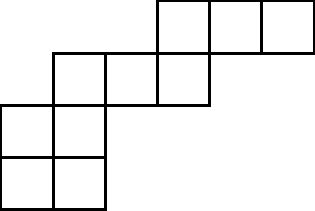}\end{array}.\] In this case, $\pat(\lambda)=\E\E\N\N\E\E\N\E\E\N$, so \cref{thm:Young} guarantees that $[\mu,\lambda]$ is an Eeta win. 
\end{example}

\Cref{thm:Young} has a somewhat surprising corollary. Namely, whether the lattice $[\mu,\lambda]\cong J(\lambda\setminus\mu)$ is an Atniss win or an Eeta win is independent of $\mu$ so long as $\lambda$ is ``deep enough'' in Young's lattice relative to $\mu$. This is actually a special case of the following much more general result. Let us write $\max(P)$ for the set of maximal elements of a poset $P$. 

\begin{theorem}\label{thm:deep_poset}
Let $P$ be a poset. Suppose $\delta,\lambda\in J(P)$ are such that $\delta\subseteq\lambda$ and every non-maximal element of $\delta$ is less than at least $2$ maximal elements of $\delta$. For every $\mu\in J(P)$ such that $\mu\subseteq \delta\setminus\max(\delta)$, the lattice $J(\lambda\setminus\mu)$ is an Atniss win if and only if the lattice $J(\lambda)$ is an Atniss win. 
\end{theorem}

Let us highlight two families of intervals in Young's lattice for which we will obtain especially nice enumerative results. Let $\rho_{a\times b}$ be the rectangular Young diagram that consists of $a$ rows of size $b$. Let $\Phi^+(A_{n})$ denote the root poset of type $A_{n}$. Then $\Phi^+(A_{n})$ is isomorphic as a poset to the skew shape $\rho_{n\times n}\setminus\delta_{n-1}$. 

\begin{theorem}\label{thm:rectangle_enumeration}
We have \[\sum_{a\geq 0}\sum_{b\geq 0}|\Eeta(J(\rho_{a\times b}))|x^by^a=\frac{(1+x)(1+y)}{1-(1+x)y^2-(1+y)x^2}.\]
\end{theorem}

\begin{theorem}\label{thm:type-A-root}
We have \[\sum_{n\geq 1}|\Eeta(J(\Phi^+(A_{n})))|z^n=\frac{-1-2z+\sqrt{z^2-4z+2-2\sqrt{1-4z+4z^2-4z^3}}}{2z}.\]  Consequently, \[|\Eeta(J(\Phi^+(A_{n})))|\sim \frac{\gamma}{\sqrt{\pi}}n^{-3/2}\rho^{n+1},\] where \[\rho=\frac{6}{2-8(3\sqrt{57}-1)^{-1/3}+(3\sqrt{57}-1)^{1/3}}\approx 3.13040\] and \[\gamma=\frac{1}{4\sqrt{291}}\sqrt{576 + (1726130304 - 69393024 \sqrt{57})^{1/3} + 
 12 (998918 + 40158 \sqrt{57})^{1/3}}\approx 0.79594.\]
\end{theorem}

Since $|J(\Phi^+(A_{n}))|$ is the $(n+1)$-th Catalan number (which grows as $(4-o(1))^n$), the preceding theorem shows that $|\Eeta(J(\Phi^+(A_n)))|/|J(\Phi^+(A_n))|$ is decays exponentially in $n$.

\subsection{Tamari Lattices}\label{subsec:tamari}
Let $\Tam_n$ denote the $n$-th Tamari lattice. These lattices, which were introduced by Tamari \cite{Tamari} in 1962, are fundamental objects in algebraic combinatorics with connections to several other areas \cite{TamariBook}; they differ from the lattices discussed in the previous subsection because they are not distributive. We will characterize Eeta wins in Tamari lattices in \Cref{prop:Tamari_product,prop:Tamari_characterization}, and this will lead to the following exact enumeration.  

\begin{theorem}\label{thm:Tamari}
The generating function $F(z)=\sum_{n\geq 1}|\Eeta(\Tam_n)|z^n$ is algebraic of degree $4$: it satisfies the equation $Q(F(z),z)=0$, where \[Q(y,z)=z + (-1 + 3 z + z^2) y + (-2 + 2 z + 3 z^2) y^2 + 3 z^2 y^3 + z^2 y^4.\] Consequently, \[|\Eeta(\Tam_n)|\sim \frac{\gamma}{\sqrt{\pi}}n^{-3/2}\rho^n,\] where $\rho\approx 2.90511$ is the unique positive real root of the polynomial \[32z^7-32z^6-155z^5-20z^4-148z^3+60z^2-8z-4\] and $\gamma\approx 1.04240$ is a root of the polynomial \begin{align*}
&\hphantom{-}\,\,17348952064 z^{14}- 11927404544 z^{12}- 
 6678731520 z^{10} \\ 
 &- 886278144 z^8- 33824320 z^6- 516144 z^4+ 4048 z^2+11.
 \end{align*}
\end{theorem}

Since $|\Tam_n|$ is the $n$-th Catalan number (which grows as $(4-o(1))^n$), the preceding theorem shows that $|\Eeta(\Tam_n)|/\Tam_n|$ is decays exponentially in $n$. 

\subsection{Outline} 
In \Cref{sec:basics}, we discuss some basic properties of lattices and Ungar moves. \Cref{sec:Weak} concerns the weak order on $S_n$; it is in this section that we establish \Cref{thm:Weak}. In \Cref{sec:young}, we prove the results from \Cref{subsec:Young} about Young's lattice. 
\Cref{sec:tamari} is devoted to analyzing Ungar games on principal order ideals of Tamari lattices; it is in this section that we prove \Cref{thm:Tamari}.  Finally, in \Cref{sec:open}, we mention potential directions for future research; we also give a short argument showing that Ungar games are $\mathsf{NC}^1$-hard.

\section{Basics}\label{sec:basics}

We assume familiarity with the theory of posets (partially ordered sets); a standard reference is \cite[Chapter~3]{Stanley}. As mentioned in \Cref{sec:intro}, we assume that every poset $P$ in this article is such that $\{y\in P:y\leq x\}$ is finite for every $x\in P$.   

Let $P$ be a poset. We tacitly view subsets of $P$ as subposets of $P$. If $u,v\in P$ are such that $u\leq v$, then the \dfn{interval} from $u$ to $v$ is the set $[u,v]=\{w\in P:u\leq w\leq v\}$. If $|[u,v]|=2$, then we say $v$ \dfn{covers} $u$. For $x\in P$, we write $\cov_P(x)$ for the set of elements of $P$ that $x$ covers. We write $\max(P)$ for the set of maximal elements of $P$. An \dfn{order ideal} of $P$ is a subset $I\subseteq P$ such that if $x,y\in P$ are such that $x\leq y$ and $y\in I$, then $x\in I$. An order ideal is \dfn{principal} if it is of the form $\{y\in P:y\leq x\}$ for some $x\in P$. Let $J(P)$ denote the set of finite order ideals of $P$, ordered by containment.  

A \dfn{meet-semilattice} is a poset $L$ such that any two elements $x,y\in L$ have a greatest lower bound, which is called their \dfn{meet} and denoted $x\wedge y$. Because the meet operation is commutative and associative, it makes sense to write $\bigwedge X$ for the meet of a nonempty finite set $X\subseteq L$. Our running assumption about posets (that principal order ideals are finite) guarantees that $L$ has a unique minimal element $\hat 0$. We say $L$ is a \dfn{lattice} if any two elements $x,y\in L$ also have a least upper bound, which is called their \dfn{join} and denoted $x\vee y$. If $L$ is a finite lattice, then it has a unique maximal element $\hat 1$. 

If $P$ is a poset, then $J(P)$ is a lattice whose meet and join operations are given by intersection and union, respectively. A finite lattice is \dfn{distributive} if it is isomorphic to $J(P)$ for some finite poset $P$. Ungar moves in distributive lattices have a simple description. For each $I\in J(P)$, we have $\cov_{J(P)}(I)=\{I\setminus\{x\}:x\in\max(I)\}$. Thus, applying an Ungar move to $I$ results in an order ideal $I\setminus T$ for some $T\subseteq\max(I)$. 

If $L$ is a meet-semilattice, then every element $x \in L$ is either an Eeta win or an Atniss win.  If $x$ is an Atniss win, then there is a nontrivial Ungar move that sends $x$ to an Eeta win.  This yields the following lemma, which we record for future reference.

\begin{lemma}\label{lem:toP2}
Let $L$ be a meet-semilattice. For every $x\in L$, the set $\MM(x)\cap\Eeta(L)$ is nonempty. 
\end{lemma}

We denote the Cartesian product of sets $X_1,\ldots,X_m$ by $X_1\times\cdots\times X_m$. If $L_1, \ldots, L_m$ are lattices, then there is a natural partial order on $L_1\times\cdots\times L_m$ in which $(x_1,\ldots,x_m)\leq(y_1,\ldots,y_m)$ if and only if $x_i\leq y_i$ for all $1\leq i\leq m$; this turns $L_1\times\cdots\times L_m$ into a lattice called the \dfn{product} of $L_1,\ldots,L_m$. The following simple lemma, which allows us to analyze the Ungar game on $L_1\times\cdots\times L_m$ in terms of the Ungar games on $L_1,\ldots,L_m$, will be very useful for us in the sequel.

\begin{lemma}\label{lem:product}
Let $L_1,\ldots,L_m$ be lattices. An element $(x_1,\ldots,x_m)$ is an Eeta win in the product $L_1\times\cdots\times L_m$ if and only if $x_i$ is an Eeta win in $L_i$ for every $i\in[m]$. That is, \[\Eeta(L_1\times\cdots\times L_m)=\Eeta(L_1)\times\cdots\times\Eeta(L_m).\]
\end{lemma}

\begin{proof}
We proceed by induction on $L_1\times\cdots\times L_m$. Choose $(x_1,\ldots,x_m)\in L_1\times\cdots\times L_m$. The key observation is that $\MM(x_1, \ldots, x_m)=\MM(x_1) \times \cdots \times \MM(x_m)$.  Thus, applying a nontrivial Ungar move to $(x_1, \ldots, x_m)$ consists of applying Ungar moves to $x_1, \ldots, x_m$ individually, where at least one of these Ungar moves is nontrivial.

Suppose $(x_1, \ldots, x_m)$ is such that $x_i \in \Eeta(L_i)$ for all $i$. Applying any nontrivial Ungar move to $(x_1, \ldots, x_m)$ produces an element $(y_1, \ldots, y_m)$ such that $y_i \in \Atniss(L_i)$ for some $i$; by the induction hypothesis, $(y_1, \ldots, y_m)\in\Atniss(L_1\times\cdots\times L_m)$. This shows that $(x_1, \ldots, x_m)\in\Eeta(L_1\times\cdots\times L_m)$. 

To prove the reverse direction, suppose $(x_1, \ldots, x_m)$ is such that $x_i\in\Atniss(L_i)$ for some $i$. Let $K \subseteq [m]$ be the (necessarily nonempty) set of indices $i$ such that $x_i \in \Atniss(L_i)$.  For each $i \in K$, there is some $y_i \in (\MM(x_i) \setminus \{x_i\}) \cap \Eeta(L_i)$.  For $i \notin K$, set $y_i=x_i$.  Then $(y_1, \ldots, y_m) \in \MM(x_1, \ldots, x_m) \setminus \{(x_1, \ldots, x_m)\}$ is an Eeta win by induction, so $(x_1, \ldots, x_m)\in\Atniss(L_1\times\cdots\times L_m)$.
\end{proof}

\section{The Weak Order}\label{sec:Weak}
Consider the symmetric group $S_n$, whose elements are the permutations of $[n]$. An \dfn{inversion} of a permutation $w\in S_n$ is a pair $(i,j)$ such that $1\leq i<j\leq n$ and $w^{-1}(i)>w^{-1}(j)$. The (right) weak order is the partial order on $S_n$ in which $u\leq v$ if and only if every inversion of $u$ is also an inversion of $v$. It is well known that the weak order on $S_n$ is a lattice. We will henceforth simply write $S_n$ for this lattice. 

The Ungar moves on $S_n$ are precisely those described in \Cref{subsec:UngarMoves}: each such move reverses some disjoint consecutive decreasing subsequences of a permutation. 

Given a word $x$ of length $k$ whose entries are distinct positive integers, we define the \dfn{standardization} of $x$ to be the permutation in $S_k$ obtained by replacing the $i$-th smallest entry in $x$ with $i$ for all $i\in[n]$. For example, the standardization of $36582$ is $24351$. Given $v\in S_k$, we say a permutation $w\in S_n$ \dfn{consecutively contains} $v$ if there exists an index $i\in[n-k+1]$ such that the standardization of $w(i)w(i+1)\cdots w(i+k-1)$ is $v$. For example, $w$ consecutively contains $1324$ if and only if there exists $i\in[n-3]$ such that $w(i)<w(i+2)<w(i+1)<w(i+3)$. We say $w$ \dfn{consecutively avoids} $v$ if $w$ does not consecutively contain $v$. 

The following lemma will allow us to prove \Cref{thm:Weak}, which tells us that as $n\to\infty$, asymptotically almost all permutations in $S_n$ are Atniss wins. Rather than demonstrate explicit winning strategies for Atniss, we will employ a strategy-stealing argument.

\begin{lemma}\label{lem:pattern_avoidance}
Let $B=\{1324,14325,154326,1654327,\ldots\}$ be the set of permutations of the form $1(m-1)(m-2)\cdots 2\,m$ for $m\geq 4$. If $w\in S_n$ is a permutation that consecutively contains one of the permutations in $B$, then $w$ is an Atniss win in $S_n$. 
\end{lemma}

\begin{proof}
Suppose $m\geq 4$ and $i\in[n-m+1]$ are such that $w(i)w(i+1)\cdots w(i+m-1)$ has standardization $1(m-1)(m-2)\cdots 2\,m$. Let $v$ be the permutation obtained from $w$ by reversing the consecutive decreasing subsequence $w(i+1)w(i+2)\cdots w(i+m-2)$. The maximal consecutive decreasing subsequences of $v$ are exactly the same as the maximal consecutive decreasing subsequences of $w$ other than $w(i+1)w(i+2)\cdots w(i+m-2)$. Therefore, $\MM(v)$ is equal to the set of permutations that can be obtained by applying an Ungar move to $w$ that involves reversing the subsequence $w(i+1)w(i+2)\cdots w(i+m-2)$. We know by \Cref{lem:toP2} that there exists an Eeta win in $\MM(v)$. This Eeta win is in $\MM(w)\setminus\{w\}$, so $w$ is an Atniss win. 
\end{proof}

Another way of phrasing the above proof of \Cref{lem:pattern_avoidance} is that Atniss can reverse the run $w(i+1)w(i+2)\cdots w(i+m-2)$ as a ``throwaway'' move and then choose whether or not to play further. 

 \begin{proof}[Proof of \Cref{thm:Weak}]
It follows from \Cref{lem:pattern_avoidance} that every Eeta win in $S_n$ consecutively avoids $1324$. It is known (see \cite{OEISA113228}) that the number of permutations in $S_n$ that consecutively avoid $1324$ is $O(0.95586^nn!)$.  
 \end{proof}

\section{Intervals in Young's Lattice} \label{sec:young}

\subsection{Intervals in Distributive Lattices}
Before we specialize our attention to Young's lattice, let us prove \Cref{thm:deep_poset}, which is much more general in scope because it deals with arbitrary finite distributive lattices. We begin with a simple but useful lemma that is analogous to \Cref{lem:pattern_avoidance}. 

\begin{lemma}\label{lem:deep_Atniss}
Let $P$ be a poset. Suppose $\lambda\in J(P)$ and $x\in \max(\lambda)$ are such that \[\max(\lambda\setminus\{x\})=\max(\lambda)\setminus\{x\}.\] Then $\lambda\in\Atniss(J(P))$. 
\end{lemma}

\begin{proof}
By \Cref{lem:toP2}, there exists $\nu\in\MM(\lambda\setminus\{x\})\cap\Eeta(J(P))$. Then $\nu=\lambda\setminus(\{x\}\cup T)$ for some $T\subseteq\max(\lambda\setminus\{x\})=\max(\lambda)\setminus\{x\}$. Since $(\{x\}\cup T)\subseteq \max(\lambda)$, we have $\nu\in\MM(\lambda)\setminus\{\lambda\}$. Because $\nu\in\Eeta(J(P))$, this proves that $\lambda\in\Atniss(J(P))$. 
\end{proof}

\begin{proof}[Proof of \Cref{thm:deep_poset}]
Let $P$ be a poset, and suppose $\delta,\lambda,\mu\in J(P)$ are such that every non-maximal element of $\delta$ is less than at least $2$ maximal elements of $\delta$ and $\mu\subseteq(\delta\setminus\max(\delta))\subseteq\delta\subseteq\lambda$. For $I\in J(P)$, let $\widetilde I=I\setminus\mu$. We will prove by induction on $|\lambda|$ that $\lambda\in\Atniss(J(P))$ if and only if $\widetilde\lambda\in\Atniss(J(\widetilde P))$. 

First, suppose there exists $x\in\max(\delta)\cap\max(\lambda)$. It follows from our hypotheses on $\delta$ and $\lambda$ that $\max(\lambda\setminus\{x\})=\max(\lambda)\setminus\{x\}$, so \Cref{lem:deep_Atniss} guarantees that $\lambda\in\Atniss(J(P))$. On the other hand, the hypothesis that $\mu\subseteq(\delta\setminus\max(\delta))$ implies that $x\in\max(\widetilde\lambda)$ and $\max(\widetilde\lambda\setminus\{x\})=\max(\widetilde\lambda)\setminus\{x\}$. Appealing to \Cref{lem:deep_Atniss} again, we find that $\widetilde\lambda\in\Atniss(J(\widetilde P))$ as well. 

Next, suppose we have $\max(\delta)\cap\max(\lambda)=\emptyset$. Because $\max(\lambda)=\max(\widetilde\lambda)$, we know that ${\MM(\widetilde\lambda)=\{\widetilde\nu:\nu\in\MM(\lambda)\}}$. For each $\nu\in\MM(\lambda)\setminus\{\lambda\}$, we have $\delta\subseteq\nu$, so we know by induction that $\nu\in\Atniss(J(P))$ if and only if $\widetilde\nu\in\Atniss(J(\widetilde P))$. This implies that there exists an element of $\Eeta(J(P))$ in $\MM(\lambda)\setminus\{\lambda\}$ if and only if there exists an element of $\Eeta(J(\widetilde P))$ in $\MM(\widetilde \lambda)\setminus\{\widetilde\lambda\}$. In other words, $\lambda\in\Atniss(J(P))$ if and only if $\widetilde\lambda\in\Atniss(J(\widetilde P))$. 
\end{proof}

\subsection{Atniss and Eeta Wins in Young's Lattice}

Note that every non-maximal element of the staircase partition $\delta_{n+1}$ is less than at least $2$ maximal elements of $\delta_{n+1}$. Moreover, we have $\delta_{n+1}\setminus\max(\delta_{n+1})=\delta_n$. Appealing to \Cref{thm:deep_poset}, we find that in order to prove \Cref{thm:Young}, it suffices to prove it when $\mu=\emptyset$. 

We will find it helpful to define a {\color{blue}\reflectbox{L}}\dfn{-path}\footnote{The symbol \reflectbox{L} is pronounced ``le'' (or ``lle'').} to be a lattice path of the form $\E^a\N^b$ for some positive integers $a$ and $b$. Given parities $\alpha$ and $\beta$, we say that such a lattice path is \dfn{$(\alpha$, $\beta)$} if $a$ is $\alpha$ and $b$ is $\beta$. For example, the \reflectbox{L}-path $\E^2\N^5$ is (even,\,odd). 
The maximal elements of a partition $\lambda$ are called the \dfn{corners} of $\lambda$. If $\lambda$ has $k$ corners, then $\pat(\lambda)$ can be written uniquely in the form $\reflectbox{\text{L}}_1\cdots\reflectbox{\text{L}}_k$, where $\reflectbox{\text{L}}_1,\ldots,\reflectbox{\text{L}}_k$ are \reflectbox{L}-paths; we call these the \dfn{maximal} {\color{blue}\reflectbox{L}}\dfn{-paths} of $\lambda$. 

\begin{proof}[Proof of \Cref{thm:Young}]
As mentioned above, we may assume $\mu=\emptyset$. Then $n=0$. Let $\lambda$ be a partition, and let $\reflectbox{\text{L}}_1,\ldots,\reflectbox{\text{L}}_k$ be the maximal \reflectbox{L}-paths of $\lambda$ so that $\pat(\lambda)=\reflectbox{\text{L}}_1\cdots\reflectbox{\text{L}}_k$. Let $c_1,\ldots,c_k$ be the corners of $\lambda$, listed from southwest to northeast (so $c_i$ corresponds naturally to $\reflectbox{\text{L}}_i$). Our goal is to show that $\lambda$ is an Eeta win in Young's lattice if and only if none of its maximal \reflectbox{L}-paths are (odd,\,odd). If $\lambda=\emptyset$, then this is vacuously true because $\lambda$ has no maximal \reflectbox{L}-paths. Thus, we may assume $\lambda$ is nonempty and proceed by induction on Young's lattice. Observe that $\lambda$ is an Atniss win in Young's lattice if and only if the transpose of $\lambda$ is an Atniss win in Young's lattice. 

First, suppose none of the maximal \reflectbox{L}-paths of $\lambda$ are (odd,\,odd). Consider $\nu\in\MM(\lambda)\setminus\{\lambda\}$. Then $\nu=\lambda\setminus T$, where $T\subseteq\{c_1,\ldots,c_k\}$ is nonempty. Let us write $T=\{c_{i_1},\ldots,c_{i_m}\}$, where $i_1<\cdots<i_m$. We may assume without loss of generality that at least one of $\reflectbox{\text{L}}_{i_1},\ldots,\reflectbox{\text{L}}_{i_m}$ is (even,\,odd) or (even,\,even); if not, then simply replace $\lambda$ and $\nu$ by their transposes. Let $j$ be the smallest index such that $c_{i_j}$ is (even,\,odd) or (even,\,even). Say $\reflectbox{L}_{i_j}=\E^a\N^b$. When we delete the corners in $T$ to obtain $\nu$, $\reflectbox{L}_{i_j}$ transforms into $\E^{a-1}\N\E\N^{b-1}$. If $i_j=1$ or $c_{i_j-1}\not\in T$, then it is straightforward to see that $\E^{a-1}\N$ is an (odd,\,odd) maximal \reflectbox{L}-path of $\nu$. If instead $i_j>1$ and $c_{i_j-1}\in T$ (so $i_j-1=i_{j-1}$), then $\reflectbox{\text{L}}_{i_j-1}$ is (odd,\,even), so it contains at least $2$ north steps. This implies that $\E^{a-1}\N$ is an (odd,\,odd) maximal \reflectbox{L}-path of $\nu$ in this case as well. In either case, we have shown that $\nu$ has an (odd,\,odd) maximal \reflectbox{L}-path, so we can use our induction hypothesis to see that $\nu$ in an Atniss win in Young's lattice. As $\nu$ was an arbitrary element of $\MM(\lambda)\setminus\{\lambda\}$, this proves that $\lambda$ is an Eeta win in Young's lattice. 

To prove the converse, suppose $\lambda$ has at least one (odd,\,odd) maximal \reflectbox{L}-path.  
We consider a few cases. 

\medskip 

\noindent {\bf Case 1.} Suppose $\reflectbox{\text{L}}_k$ has at least $2$ north steps. Let $\lambda^{\#}$
 be the partition obtained by removing the first two rows from $\lambda$. Then $\lambda^{\#}$ has at least one (odd,\,odd) maximal \reflectbox{L}-path, so it is an Atniss win in Young's lattice by induction. This means that there is a nonempty set $T^{\#}$ of corners of $\lambda^{\#}$ such that $\lambda^{\#}\setminus T^{\#}$ is an Eeta win. By induction, $\lambda^{\#}\setminus T^{\#}$ has no (odd,\,odd) maximal \reflectbox{L}-paths. The set $T^{\#}$ corresponds naturally to a set $T$ of corners of $\lambda$, and $\lambda\setminus T$ is the partition obtained by adding the first two rows of $\lambda$ to the top of $\lambda^{\#}\setminus T^{\#}$. Then $\lambda\setminus T$ has no (odd,\,odd) maximal \reflectbox{L}-paths, so it is an Eeta win by induction. Since $(\lambda\setminus T)\in\MM(\lambda)\setminus\{\lambda\}$, this shows that $\lambda$ is an Atniss win. 

 \medskip 

 \noindent {\bf Case 2.} Suppose $\reflectbox{\text{L}}_k=\E\N$. In this case, $\max(\lambda\setminus\{c_k\})=\max(\lambda)\setminus\{c_k\}$. Setting $P=\mathbb{N}^2$ in \Cref{lem:deep_Atniss}, we find that $\lambda$ is an Atniss win.

\medskip 

\noindent {\bf Case 3.} Suppose $\reflectbox{\text{L}}_k=\E^a\N$ for some $a\geq 2$.  Let $\lambda^{\#}$ be the partition obtained by removing the first row from $\lambda$. By \Cref{lem:toP2}, there is a (possibly empty) set $T^{\#}$ of corners of $\lambda^{\#}$ such that $\lambda^{\#}\setminus T^{\#}$ is an Eeta win. By induction, $\lambda^{\#}\setminus T^{\#}$ has no (odd,\,odd) maximal \reflectbox{L}-paths. The set $T^{\#}$ corresponds naturally to a set $T$ of corners of $\lambda$, and $\lambda\setminus T$ is the partition obtained by adding the first row of $\lambda$ to the top of $\lambda^{\#}\setminus T^{\#}$. Then $\lambda\setminus T$ has no (odd,\,odd) maximal \reflectbox{L}-paths except for possibly the northeastmost \reflectbox{L}-path (call this $\widetilde{\reflectbox{\text{L}}}$).  Notice that $\widetilde{\reflectbox{\text{L}}} \in \{\E^a \N, \E^{a+1} \N\}$.  If $\widetilde{\reflectbox{\text{L}}}$ is (odd,\,odd), then $\lambda\setminus (\{c_k\}\cup T)$ has no (odd,\,odd) maximal \reflectbox{L}-paths and hence is an Eeta win by induction.  Since $(\lambda\setminus(\{c_k\}\cup T))\in\MM(\lambda)\setminus\{\lambda\}$, this shows that $\lambda$ is an Atniss win if $\widetilde{\reflectbox{\text{L}}}$ is (odd,\,odd).  Now suppose $\widetilde{\reflectbox{\text{L}}}$ is instead (even,\,odd).  Notice that $T$ is nonempty since, if it were empty, then $\lambda \setminus T=\lambda$ would have no (odd,\,odd) maximal \reflectbox{L}-paths, contrary to our standing assumption.  So $T$ is nonempty, and, since $\lambda\setminus T$ has no (odd,\,odd) maximal \reflectbox{L}-paths, again we are done by induction.
\end{proof}

\subsection{Rectangles}
Let us now prove \Cref{thm:rectangle_enumeration}, which enumerates Eeta wins in $J(\rho_{a\times b})$, where $\rho_{a\times b}$ is the $a\times b$ rectangle poset. 

\begin{proof}[Proof of \Cref{thm:rectangle_enumeration}]
For fixed $a,b\geq 0$, we can append extra north steps to the beginning and extra east steps to the end of the path associated to an order ideal in $J(\rho_{a\times b})$ so that the resulting path uses a total of $a$ north steps and $b$ east steps. Then such a path
can be written uniquely in the form $\N^s\reflectbox{\text{L}}_1\cdots\reflectbox{\text{L}}_k\E^t$, where $s,t\geq 0$ and $\reflectbox{\text{L}}_1,\ldots,\reflectbox{\text{L}}_k$ are \reflectbox{L}-paths that use a total of $a-s$ north steps and $b-t$ east steps. It follows from \Cref{thm:Young} that such an order ideal is an Eeta win in $J(\rho_{a\times b})$ if and only if none of $\reflectbox{\text{L}}_1,\ldots,\reflectbox{\text{L}}_k$ are (odd,\,odd). We will consider generating functions that count \reflectbox{L}-paths, with the variable $x$ keeping track of the number of east steps and the variable $y$ keeping track of the number of north steps. The generating function for (odd,\,odd) \reflectbox{L}-paths is \[(x+x^3+x^5+\cdots)(y+y^3+y^5+\cdots)=\frac{xy}{(1-x^2)(1-y^2)},\] so the generating function for \reflectbox{L}-paths that are not (odd,\,odd) is \[(x+x^2+x^3+\cdots)(y+y^2+y^3+\cdots)-\frac{xy}{(1-x^2)(1-y^2)}=\frac{xy}{(1-x)(1-y)}-\frac{xy}{(1-x^2)(1-y^2)}.\] The generating function that counts sequences of \reflectbox{L}-paths that are not (odd,\,odd) is then \[\frac{1}{1-\left(\frac{xy}{(1-x)(1-y)}-\frac{xy}{(1-x^2)(1-y^2)}\right)}.\] Hence, 
\begin{align*}
\sum_{a\geq 0}\sum_{b\geq 0}|\Eeta(J(\rho_{a\times b}))|x^by^a&=\sum_{s\geq 0}y^s\sum_{t\geq 0}x^t\cdot\frac{1}{1-\left(\frac{xy}{(1-x)(1-y)}-\frac{xy}{(1-x^2)(1-y^2)}\right)} \\ &=\frac{1}{(1-x)(1-y)}\cdot\frac{1}{1-\left(\frac{xy}{(1-x)(1-y)}-\frac{xy}{(1-x^2)(1-y^2)}\right)} \\ 
&=\frac{(1+x)(1+y)}{1-(1+x)y^2-(1+y)x^2}. \qedhere
\end{align*}
\end{proof}

\subsection{Type-$A$ Root Posets}
The root poset $\Phi^+(A_{n})$---which is isomorphic to the skew shape $\rho_{n\times n}\setminus\delta_{n-1}$---is an important poset in algebraic combinatorics with several interesting properties. For example, the number of order ideals of $\Phi^+(A_{n})$ is the Catalan number $C_{n+1}=\frac{1}{n+2}\binom{2(n+1)}{n+1}$. In this subsection, we prove \Cref{thm:type-A-root}, which enumerates Eeta wins in $J(\Phi^+(A_{n}))$. 

As before, we view an order ideal in $J(\rho_{n\times n}\setminus\delta_{n-1})$ as a skew shape $\lambda\setminus\delta_{n-1}$ such that $\lambda\subseteq\rho_{n\times n}$, and we consider the associated lattice path $\pat(\lambda)$. Suppose $\pat(\lambda)$ uses $s$ north steps and $t$ east steps. Then $s,t\in\{n-1,n\}$. Let $\pat'(\lambda)=\N^{n-s}\pat(\lambda)\,\E^{n-t}$. If we delete from $\pat'(\lambda)$ all steps that lie on the boundary of $\delta_{n-1}$ or on the $x$-axis or $y$-axis, then we will break $\pat'(\lambda)$ into lattice paths $\eta^{(1)},\ldots,\eta^{(r)}$ that represent order ideals of smaller type-$A$ root posets. That is, for each $1\leq i\leq r$, there is a positive integer $n_i$ such that $\eta^{(i)}=\pat(\lambda^{(i)})$ for some partition $\lambda^{(i)}$ satisfying $\delta_{n_i-1}\subseteq\lambda^{(i)}\subseteq\rho_{n_i\times n_i}$. In fact, this construction is designed so that $\lambda^{(i)}$ contains the slightly larger staircase $\delta_{n_i}$. Setting $\mu=\delta_{n_i-1}$ in \Cref{thm:Young}, we find that the interval $[\delta_{n_i-1},\lambda^{(i)}]$ is an Eeta win if and only if $\eta^{(i)}$ does not contain an odd-length block of east steps immediately followed by an odd-length block of north steps.   
It is straightforward to see that \[J(\lambda\setminus\delta_{n-1})\cong [\delta_{n_1-1},\lambda^{(1)}]\times\cdots\times [\delta_{n_r-1},\lambda^{(r)}],\] so it follows from \Cref{lem:product} that $\lambda\setminus\delta_{n-1}$ is an Eeta win in $J(\rho_{n\times n}\setminus\delta_{n-1})$ if and only if none of $\eta^{(1)},\ldots,\eta^{(r)}$ contains an odd-length block of east steps immediately followed by an odd-length block of north steps. 

\begin{example}\label{exam:typeA1}
Let $n=12$, and let $\lambda=(11,11,11,10,10,6,6,4,3,3,3,1)$. Then \[\pat'(\lambda)={\color{DarkGreen}\E\N\E\E\N\N}{\color{red}\N\E\N\E}{\color{SkyBlue}\E\N}{\color{red}\N\E}{\color{MyPurple}\E\E\E\N\N\E\N\N}{\color{red}\N\E}\] is drawn in \Cref{fig:lattice_path_decomposition}. The steps lying on the boundary of $\delta_{10}$ or the $x$-axis or $y$-axis are colored red. If we delete those steps, then we are left with the lattice paths \[\eta^{(1)}={\color{DarkGreen}\E\N\E\E\N\N},\quad\eta^{(2)}={\color{SkyBlue}\E\N},\quad\eta^{(3)}={\color{MyPurple}\E\E\E\N\N\E\N\N}.\] Then $n_1=3$, $n_2=1$, and $n_3=4$. The corresponding partitions are \[\lambda^{(1)}={\color{DarkGreen}(3,3,1)},\quad\lambda^{(2)}={\color{SkyBlue}(1)},\quad\lambda^{(3)}={\color{MyPurple}(4,4,3,3)}.\] For each $1\leq i\leq 3$, the skew shape $\lambda^{(i)}\setminus\delta_{n_i-1}$ is an order ideal of $\rho_{n_i\times n_i}\setminus\delta_{n_i-1}$. Notice that each $\lambda^{(i)}$ actually contains the staircase $\delta_{n_i}$. Since the intervals $[\delta_{2},\lambda^{(1)}]$ and $[\delta_0,\lambda^{(2)}]$ are Atniss wins, the lattice $J(\lambda\setminus\delta_{10})$ is also an Atniss win. 
\end{example}

\begin{figure}[ht]
\begin{center}\hspace{-2.7cm}\includegraphics[height=4.395cm]{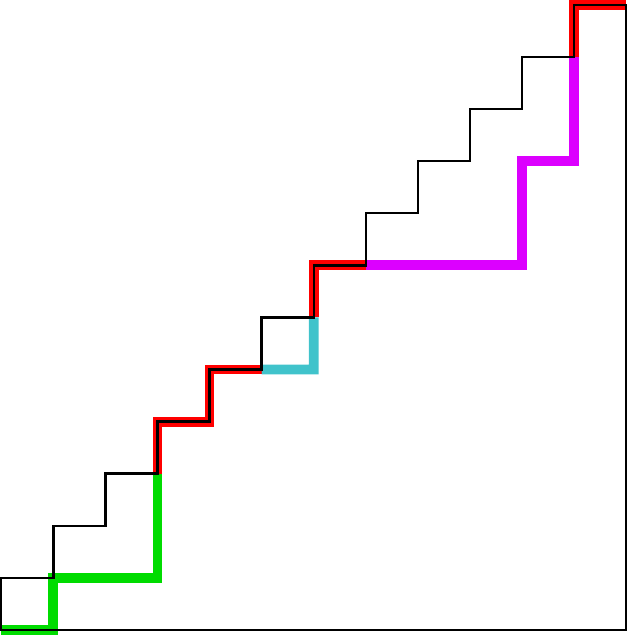}\qquad\qquad\qquad\qquad\includegraphics[height=4.395cm]{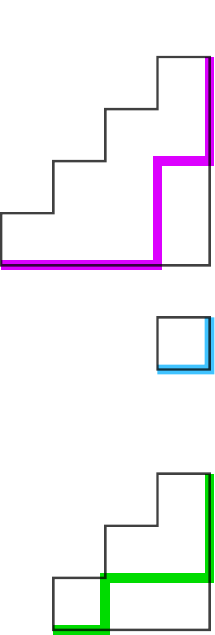}
  \end{center}
\caption{Deleting the ({\color{red}red}) steps that lie on the boundary of $\delta_{10}$ or the $x$-axis or $y$-axis breaks a lattice path into $3$ smaller lattice paths.}\label{fig:lattice_path_decomposition}
\end{figure}

In our enumeration of Eeta wins in $J(\rho_{n\times n}\setminus\delta_{n-1})$, it will be convenient to use the language of Dyck paths. A \dfn{Dyck path} of semilength $n$ is a path in $\mathbb R^2$ consisting of up (i.e., $(1,1)$) steps and down (i.e., $(1,-1)$) steps that starts at $(0,0)$, ends at $(2n,0)$, and never passes below the $x$-axis. We  can represent a Dyck path as a word over the alphabet $\{\U,\D\}$, where $\U$ stands for an up step and $\D$ stands for a down step. 

An \dfn{ascending run} (respectively, \dfn{descending run}) of a Dyck path is a maximal consecutive string of up (respectively, down) steps. Say a run is \dfn{odd} (respectively, \dfn{even}) if it has an odd (respectively, even) number of steps. Say a run is \dfn{\oddd} if it is odd and does not touch the $x$-axis or it is even and does touch the $x$-axis. Say a run is \dfn{$\odddd$} if it is odd and does not contain the first or last step of the Dyck path or it is even and contains the first or last step of the Dyck path. 

\begin{example}
Consider the Dyck path \[{\color{SkyBlue}\U}\vert\overline{{\color{SkyBlue}\D}}\vert\underline{{\color{MyPurple}\U\U}}\vert\overline{\underline{{\color{SkyBlue}\D}}}\vert\overline{\underline{{\color{SkyBlue}\U}}}\vert\underline{{\color{MyPurple}\D\D}}\vert\underline{{\color{MyPurple}\U\U\U\U}}\vert\overline{\underline{{\color{MyPurple}\D\D\D\D}}}=\begin{array}{l}\includegraphics[height=1.1cm]{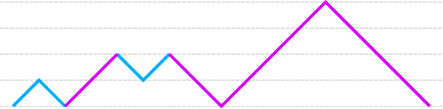}\end{array}.\] 
Odd runs are in {\color{SkyBlue}light blue}, while even runs are in {\color{MyPurple}lavender}. In the word representation of this Dyck path, we have separated the runs by bars for clarity, and we have underlined the $\oddd$ runs and overlined the $\odddd$ runs.  
\end{example} 

Given adjectives $\alpha$ and $\beta$ that describe runs, let us say a Dyck path is \dfn{$(\alpha,\beta)$-avoiding} if it does not contain an $\alpha$ ascending run immediately followed by a $\beta$ descending run. For example, a Dyck path is (odd,\,$\odddd$)-avoiding if it does not contain an odd ascending run immediately followed by an $\odddd$ descending run. 

Given an order ideal $\lambda\setminus\delta_{n-1}$ of $\rho_{n\times n}\setminus\delta_{n-1}$, let $\pat^*(\lambda)$ be the word obtained from $\pat'(\lambda)$ by replacing each $\E$ with $\U$ and replacing each $\N$ with $\D$. Then $\U\,\pat^*(\lambda)\,\D$ is a Dyck path of semilength $n+1$. For example, if $\lambda$ is the partition from \Cref{exam:typeA1}, then $\U\,\pat^*(\lambda)\,\D$ is the Dyck path \[\U{\color{DarkGreen}\U\D\U\U\D\D}{\color{red}\D\U\D\U}{\color{SkyBlue}\U\D}{\color{red}\D\U}{\color{MyPurple}\U\U\U\D\D\U\D\D}{\color{red}\D\U}\D.\]
It follows from the above discussion that $\lambda\setminus\delta_{n-1}$ is an Eeta win in $J(\rho_{n\times n}\setminus\delta_{n-1})$ if and only if $\U\,\pat^*(\lambda)\,\D$ is ($\oddd$,\,$\oddd$)-avoiding. This allows us to prove \Cref{thm:type-A-root}. 

\begin{proof}[Proof of \Cref{thm:type-A-root}]

Let $\mathcal F_n$ be the set of (odd,\,odd)-avoiding Dyck paths of semilength $n$. Let $\underline{\mathcal F}_n$ and $\overline{\mathcal F}_n$ be the set of ($\oddd$,$\oddd$)-avoiding Dyck paths of semilength $n$ and the set of ($\odddd$,\,$\odddd$)-avoiding Dyck paths of semilength $n$, respectively. Let \[F(z)=\sum_{n\geq 0}|\mathcal F_n|z^n,\quad \underline{F}(z)=\sum_{n\geq 0}|\underline{\mathcal{F}}_n|z^n,\quad \overline{F}(z)=\sum_{n\geq 0}|\overline{\mathcal{F}}_n|z^n.\] Let $\mathcal G_n$ and $\mathcal H_n$ be the set of (odd,\,$\odddd$)-avoiding Dyck paths of semilength $n$ and the set of ($\odddd$,\,odd)-avoiding Dyck paths of semilength $n$, respectively. Let \[G(z)=\sum_{n\geq 0}|\mathcal G_n|z^n\quad\text{and}\quad H(z)=\sum_{n\geq 0}|\mathcal H_n|z^n.\] 

If $\Lambda$ is a nonempty Dyck path, then there are unique Dyck paths $\Lambda'$ and $\Lambda''$ such that ${\Lambda=\U\Lambda'\D\Lambda''}$. For example, if $\Lambda=\U\U\U\D\D\U\D\D\U\D$, then $\Lambda'=\U\U\D\D\U\D$ and $\Lambda''=\U\D$. We call $\Lambda'$ and $\Lambda''$ the \dfn{primary part} of $\Lambda$ and the \dfn{secondary part} of $\Lambda$, respectively. 
A nonempty Dyck path is ($\oddd$,\,$\oddd$)-avoiding if and only if its primary part is (odd,\,odd)-avoiding and its secondary part is ($\oddd$,\,$\oddd$)-avoiding. Therefore, 
\begin{equation}\label{eq:Dyck1}
\underline F(z)-1=zF(z)\underline F(z).
\end{equation} 
A nonempty Dyck path is (odd,\,odd)-avoiding if and only if its primary part is nonempty and ($\odddd$,\,$\odddd$)-avoiding and its secondary part is (odd,\,odd)-avoiding. Therefore, 
\begin{equation}\label{eq:Dyck2}
F(z)-1=z(\overline F(z)-1)F(z).
\end{equation}
A nonempty Dyck path $\Lambda$ is ($\odddd$,\,$\odddd$)-avoiding if and only if one of the following holds: 
\begin{itemize}
\item The primary part $\Lambda'$ is (odd,\,odd)-avoiding, and the secondary part $\Lambda''$ is empty. 
\item The primary part $\Lambda'$ is (odd,\,$\odddd$)-avoiding, and the secondary part $\Lambda''$ is nonempty and (odd,\,$\odddd$)-avoiding. 
\end{itemize}  Therefore, 
\begin{equation}\label{eq:Dyck3}
\overline F(z)-1=zF(z)+zG(z)(G(z)-1).
\end{equation} 
A nonempty Dyck path $\Lambda$ is (odd,\,$\odddd$)-avoiding if and only if one of the following holds: 
\begin{itemize}
\item The primary part $\Lambda'$ is ($\odddd$,\,odd)-avoiding, and the secondary part $\Lambda''$ is empty. 
\item The primary part $\Lambda'$ is nonempty and ($\odddd$,\,$\odddd$)-avoiding, and the secondary part $\Lambda''$ is nonempty and (odd,\,$\odddd$)-avoiding. 
\end{itemize} 
Therefore, 
\begin{equation}\label{eq:Dyck4}
G(z)-1=zH(z)+z(\overline F(z)-1)(G(z)-1).
\end{equation}
There is a simple bijection $\mathcal G_n\to\mathcal H_n$ that acts by simply reversing a Dyck path and swapping $\U$'s and $\D$'s (i.e., reflecting the path through the line $x=n$), so 
\begin{equation}\label{eq:Dyck5}
G(z)=H(z).
\end{equation} 

\Cref{eq:Dyck1,eq:Dyck2,eq:Dyck3,eq:Dyck4,eq:Dyck5} form a system in the unknowns $F(z)$, $\underline F(z)$, $\overline F(z)$, $G(z), H(z)$. We can solve this system using a computer algebra program to find that \[\underline F(z)=1+z+\frac{-1-2z+\sqrt{z^2-4z+2-2\sqrt{1-4z+4z^2-4z^3}}}{2}.\] 
For $n\geq 1$, the poset $\Phi^+(A_{n})$ is isomorphic to $\rho_{n\times n}\setminus\delta_{n-1}$. As discussed above, there is a bijection from $\Eeta(J(\rho_{n\times n}\setminus\delta_{n-1}))$ to $\underline{\mathcal F}_{n+1}$ given by ${\lambda\setminus\delta_{n-1}\mapsto \U\,\pat^*(\lambda)\,\D}$. Hence, \[\sum_{n\geq 1}|\Eeta(J(\Phi^+(A_{n})))|z^n=\frac{1}{z}(-1-z+\underline F(z))=\frac{-1-2z+\sqrt{z^2-4z+2-2\sqrt{1-4z+4z^2-4z^3}}}{2z},\] as desired. 

The method used to derive the asymptotics in the statement of the theorem is routine and is discussed in \cite[Chapter~VII]{Flajolet}; we will just sketch the details. The constant $\rho$ is determined by noting that $1/\rho$ is the complex singularity of $\frac{1}{z}(-1-z+\underline F(z))$ closest to the origin (Pringsheim's theorem guarantees that $\rho$ is positive and real). One can use a computer algebra software such as Maple to expand $\frac{1}{z}(-1-z+\underline F(z))$ as a Puiseux series centered at $1/\rho$; the result is $\beta_0+{\beta_1(z-1/\rho)^{1/2}}+o((z-1/\rho)^{1/2})$ for some explicitly computable algebraic numbers $\beta_0$ and $\beta_1$. Following the discussion in \cite[Chapter~VII]{Flajolet}, this expansion transfers into an asymptotic formula of the form \[|\Eeta(J(\Phi^+(A_n)))|\sim \frac{\gamma}{\sqrt{\pi}}n^{-3/2}\rho^{n+1},\] and one can use a computer algebra software to find that $\gamma$ is as stated in the theorem.
\end{proof}

\section{Tamari Lattices}\label{sec:tamari}

A permutation $w\in S_n$ is called \dfn{$312$-avoiding} if there do not exist indices $i_1<i_2<i_3$ such that $w(i_2)<w(i_3)<w(i_1)$. The set of $312$-avoiding permutations in $S_n$ forms a sublattice of the weak order that we denote by $\Tam_n$; this is one of the many combinatorial realizations of the $n$-th \dfn{Tamari lattice}. Our goal in this section is to prove \Cref{thm:Tamari}, which enumerates Eeta wins in Tamari lattices both exactly and asymptotically. Our first order of business is to describe Ungar moves in Tamari lattices. 

Suppose $w\in S_n$. If there exist indices $i$ and $i'$ such that $i+1 < i'$ and $w(i+1)<w(i')<w(i)$, then we can perform an \dfn{allowable swap} by swapping the entries $w(i)$ and $w(i+1)$. Let $\pi_\downarrow(w)$ be the permutation obtained from $w$ by repeatedly performing allowable swaps until no more allowable swaps can be performed. The element $\pi_\downarrow(w)$ is well defined (i.e., does not depend on the sequence of allowable swaps) and is $312$-avoiding \cite{ReadingCambrian}. Hence, we obtain a map $\pi_\downarrow\colon S_n\to\Tam_n$. Note that $\pi_\downarrow(w)=w$ if and only if $w\in\Tam_n$. 

The first author showed \cite[Equation~(1)]{DefantMeeting} that applying a maximal Ungar move within $\Tam_n$ to a $312$-avoiding permutation $w$ is equivalent to applying a maximal Ungar move to $w$ within the weak order on $S_n$ and then applying $\pi_\downarrow$. The exact same argument (which we omit) shows that applying an arbitrary nontrivial Ungar move to $w$ within $\Tam_n$ is equivalent to applying an arbitrary nontrivial Ungar move to $w$ within $S_n$ and then applying $\pi_\downarrow$. In what follows, we give an equivalent description of Tamari lattice Ungar moves that will be more suitable for our purposes. 

The \dfn{plot} of a permutation $w\in S_n$ is the diagram showing the points $(i,w(i))$ for all $i\in[n]$. We often identify permutations with their plots. Suppose $u\in S_m$ and $v\in S_n$. The \dfn{direct sum} $u\oplus v$ and the \dfn{skew sum} $u\ominus v$ are the permutations in $S_{m+n}$ defined by \[(u\oplus v)(i)=\begin{cases} u(i)& \mbox{ if } 1\leq i\leq m; \\
m+v(i-m) & \mbox{ if }m+1\leq i\leq m+n \end{cases}\] and \[(u\ominus v)(i)=\begin{cases} n+u(i)& \mbox{ if } 1\leq i\leq m; \\
v(i-m) & \mbox{ if }m+1\leq i\leq m+n. \end{cases}\]
The plot of $u\oplus v$ (respectively, $u\ominus v$) is obtained by placing the plot of $v$ to the northeast (respectively, southeast) of the plot of $u$. If $U$ and $V$ are sets of permutations, then we let \[U\oplus V=\{u\oplus v:u\in U,~v\in V\}\quad\text{and}\quad U\ominus V=\{u\ominus v:u\in U,~v\in V\}.\] 

A permutation is called \dfn{decomposable} if it can be written as the direct sum of two smaller permutations; otherwise, it is \dfn{indecomposable}. Every permutation $w$ can be written uniquely in the form $u_1\oplus\cdots\oplus u_k$ for some indecomposable permutations $u_1,\ldots,u_k$; these indecomposable permutations are called the \dfn{components} of $w$. Note that a permutation is $312$-avoiding if and only if all of its components are $312$-avoiding. Moreover, a $312$-avoiding permutation in $S_n$ is indecomposable if and only if its last entry is $1$. 

Suppose $w=u_1\oplus\cdots\oplus u_k\in\Tam_n$, where $u_1,\ldots,u_k$ are the components of $w$. Applying an Ungar move to $w$ is equivalent to applying Ungar moves to $u_1,\ldots,u_k$ independently and then taking the direct sum of the resulting permutations. In symbols, \[\MM(w)=\MM(u_1)\oplus\cdots\oplus\MM(u_k).\] This shows that in order to describe Ungar moves, we can restrict our attention to indecomposable $312$-avoiding permutations. 

Suppose $w\in\Tam_n$ is indecomposable, and assume $n\geq 2$. We can write $w=w'\ominus 1$ for some $w'\in \Tam_{n-1}$. Let $v_1,\ldots,v_k$ be the components of $w'$ so that $w=(v_1\oplus\cdots\oplus v_k)\ominus 1$. Suppose $v_k\in\Tam_m$. To apply an Ungar move to $w$, we apply an Ungar move to $w'$ and then either keep the entry $1$ in the last position or slide the $1$ into position $n-m$. In symbols, we have 
\[\MM(w)=\left(\left(\MM(v_1)\oplus\cdots\oplus\MM(v_k)\right)\ominus\{1\}\right)\sqcup\left(\left(\left(\MM(v_1)\oplus\cdots\oplus\MM(v_{k-1})\right)\ominus\{1\}\right)\oplus\MM(v_k)\right).\]

\begin{example}
Suppose \[w=32568741=(21\oplus 23541)\ominus 1=\begin{array}{l}
\includegraphics[height=2.2cm]{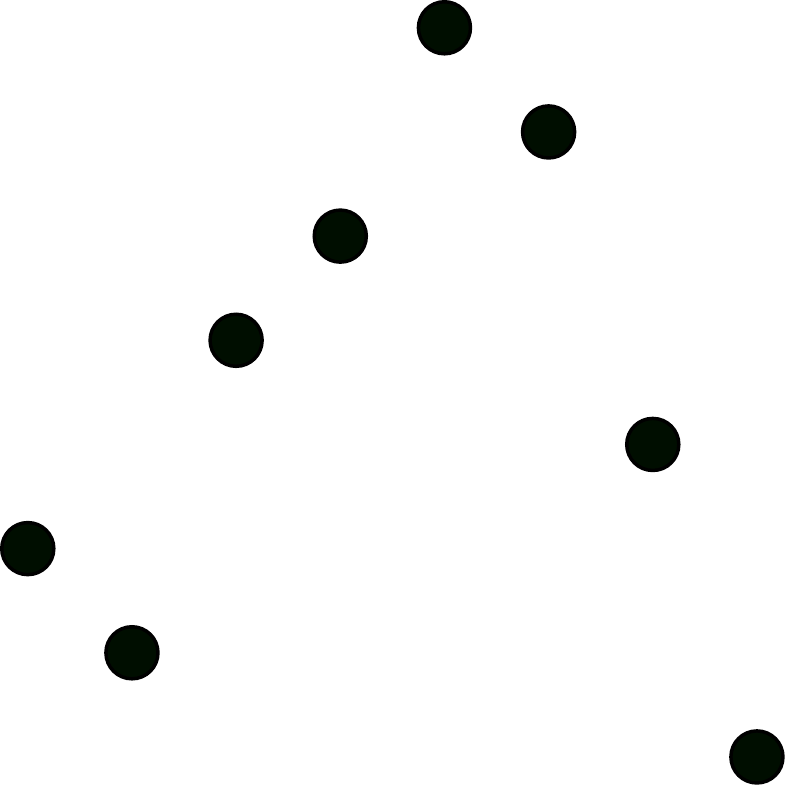}\end{array}\in\Tam_8.\] The $8$ indecomposable elements of $\MM(w)$ are 
\[\begin{array}{l}
\includegraphics[height=6.217cm]{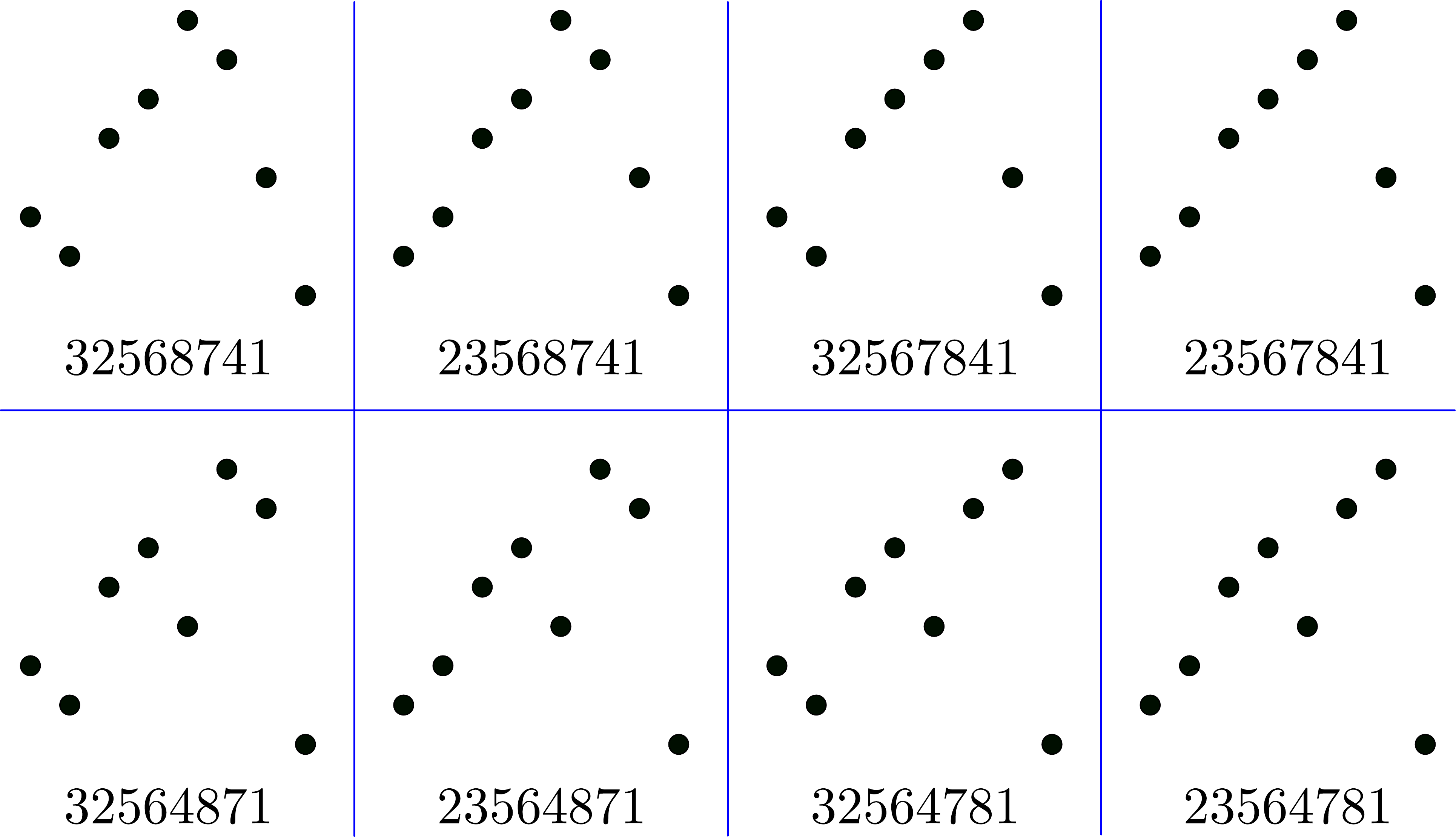}\end{array},\] while the $8$ decomposable elements of $\MM(w)$ are 
\[\begin{array}{l}
\includegraphics[height=6.217cm]{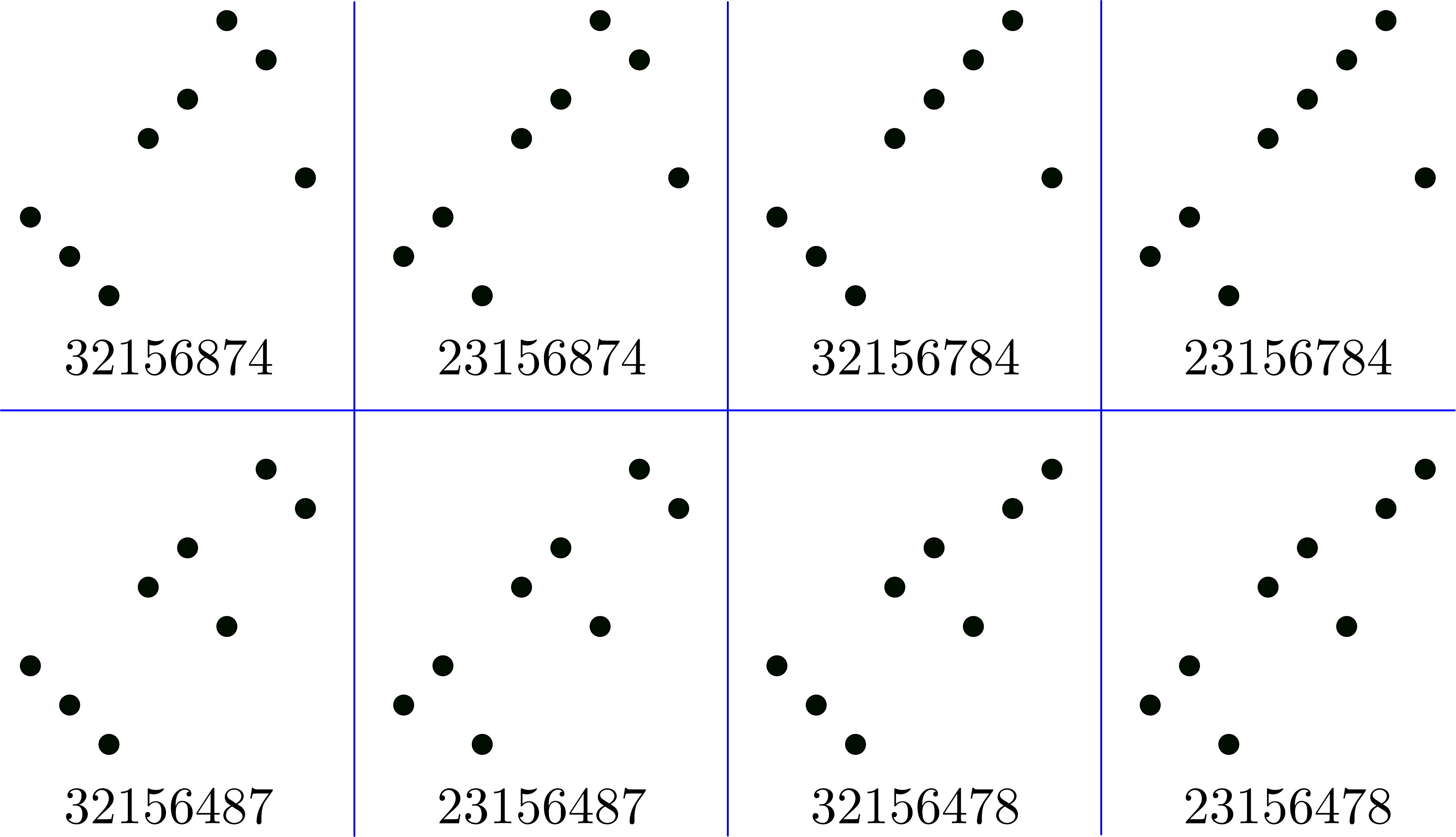}\end{array}.\] 
\end{example}

Recall from \Cref{sec:Weak} the definition of the \emph{standardization} of a word. Consider a permutation ${w=w(1)\cdots w(n)\in\Tam_n}$. Let us say $w$ is \dfn{even-districted} if one of the following conditions holds: 
\begin{itemize}
\item $n=1$; 
\item $n\geq 3$, $w=w'\ominus 1$ for some $w'\in\Tam_{n-1}$ with an even number of components, and the standardization of $w(1)\cdots w(n-2)$ is an Eeta win in $\Tam_{n-2}$.  
\end{itemize}

The following two propositions provide a recursive description of Eeta wins in Tamari lattices. 

\begin{proposition}\label{prop:Tamari_product}
Let $w=u_1\oplus\cdots\oplus u_k\in\Tam_n$, where $u_1,\ldots,u_k$ are the components of $w$. Let $n_i$ be the size of $u_i$. Then $w\in\Eeta(\Tam_n)$ if and only if $u_i\in\Eeta(\Tam_{n_i})$ for all $1\leq i\leq k$. 
\end{proposition}

\begin{proof}
The interval $[\hat 0,w]$ in $\Tam_n$ is isomorphic to the product $[\hat 0,u_1]\times\cdots\times[\hat 0,u_k]$ (abusing notation, we use $\hat 0$ to denote the bottom elements of different lattices). Therefore, the desired result follows from \Cref{lem:product}. 
\end{proof}

\begin{proposition}\label{prop:Tamari_characterization}
An indecomposable permutation is an Eeta win in $\Tam_n$ if and only if it is even-districted. 
\end{proposition}

Our proof of \Cref{prop:Tamari_characterization} will require the following lemmas. We refer the reader to \Cref{exam:Tamari1} for an illustration of the proof of \Cref{lem:ribute}.

\begin{lemma}\label{lem:ribute}
Let $x\in\Tam_n$, and suppose there exists a permutation $y\in\MM(x)\cap\Eeta(\Tam_n)$ with an even number of components. Then there exists $\widetilde y\in\MM(x)$ such that $\widetilde y\ominus 1$ is even-districted. 
\end{lemma}

\begin{proof}
Let $v\ominus 1$ be the final component of $y$. Then $y=u\oplus(v\ominus 1)$, where $u$ has an odd number of components. Let $m$ be the size of $u$ (so $u\in\Tam_m$). Because $y\in\Eeta(\Tam_n)$, we know by \Cref{prop:Tamari_product} that all of the components of $u$ are Eeta wins in their respective Tamari lattices. The number $m+1$ is the last entry in $y$. Since $y\in\MM(x)$, we know that $y\leq x$ in the weak order. This implies that $m+1$ appears to the right of the entries $m+2,\ldots,n$ in $x$. Let $r=x^{-1}(m+1)$. Because $x$ is $312$-avoiding, the entries in positions $r-(n-m)+1,\ldots,r-1$ in $x$ are the numbers $m+2,\ldots,n$ in some order; that is \[\{x(r-(n-m)+i):1\leq i\leq n-m-1\}=\{m+2,\ldots,n\}.\] Let $z\in\Tam_{n-m-1}$ be the standardization of the sequence $x(r-(n-m)+1)\cdots x(r-1)$. According to \Cref{lem:toP2}, there exists $z'\in\MM(z)\cap\Eeta(\Tam_{n-m-1})$. Note that $z'\ominus 1\in\MM(z\ominus 1)$. 

Applying an Ungar move to $x$ amounts to moving the entries $1,\ldots,m$ and then moving the entries $m+1,\ldots,n$ independently. To make this more precise, let $w\in\Tam_{m+1}$ be the permutation obtained from $x$ by deleting the entries $m+2,\ldots,n$, and let $Z$ be the set of permutations of the set $\{m+1,\ldots,n\}$ whose standardizations are in $\MM(z\ominus 1)$. Then $\MM(x)$ is the set of permutations that can be obtained by selecting a permutation $w'\in\MM(w)$ and then replacing the entry $m+1$ in $w'$ with a permutation in $Z$. Since $y=u\oplus(v\ominus 1)\in\MM(x)$, it must be the case that $u\oplus 1\in\MM(w)$. Also, there is a permutation in $Z$ whose standardization is $z'\ominus 1$. It follows that $u\oplus(z'\ominus 1)\in\MM(x)$. Let $\widetilde y=u\oplus(z'\ominus 1)$. To complete the proof, we just need to show that $\widetilde y\ominus 1$ is even-districted. 

The components of $\widetilde y$ are the components of $u$ and the indecomposable permutation $z'\ominus 1$. Since $u$ has an odd number of components, $\widetilde y$ has an even number of components. If we delete the last two entries from $\widetilde y\ominus 1$ and then standardize, we obtain $u\oplus z'$. We observed above that all of the components of $u$ are Eeta wins, and we chose $z'$ to be an Eeta win. Therefore, it follows from \Cref{prop:Tamari_product} that $u\oplus z'$ is an Eeta win. This demonstrates that $\widetilde y\ominus 1$ is even-districted. 
\end{proof}

\begin{example}\label{exam:Tamari1}
Preserve the notation from the proof of \Cref{lem:ribute}. Let $n=9$. Suppose \[x=237986541=\begin{array}{l}\includegraphics[height=2.492cm]{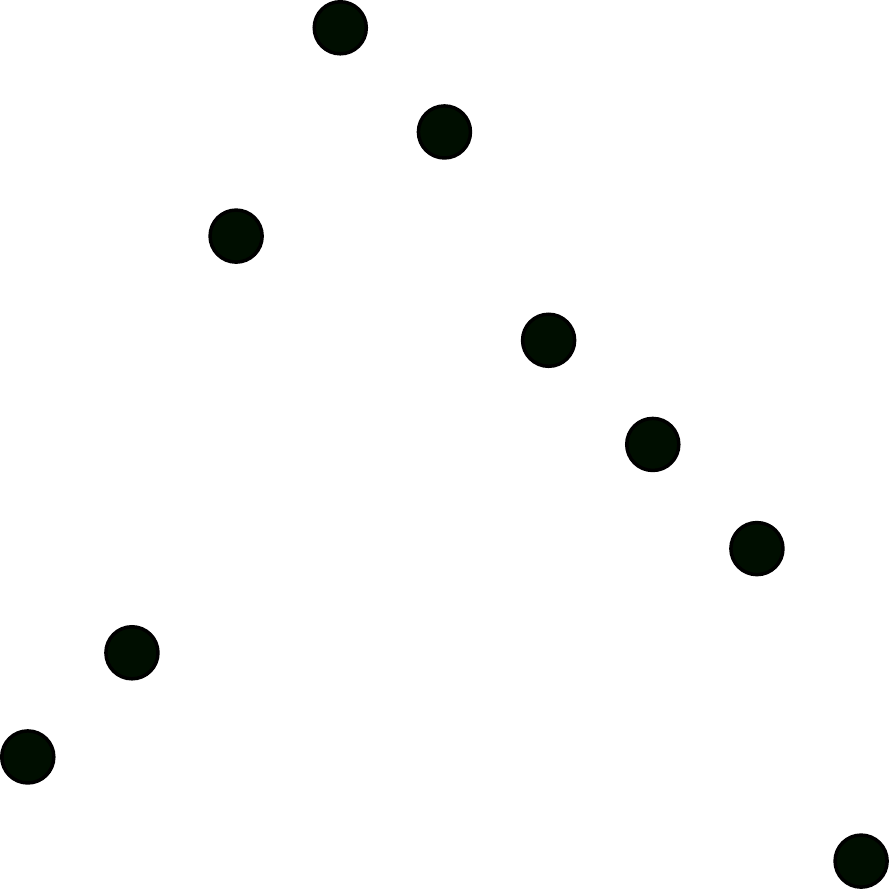}\end{array}\quad \text{and}\quad y=231457986=\begin{array}{l}\includegraphics[height=2.492cm]{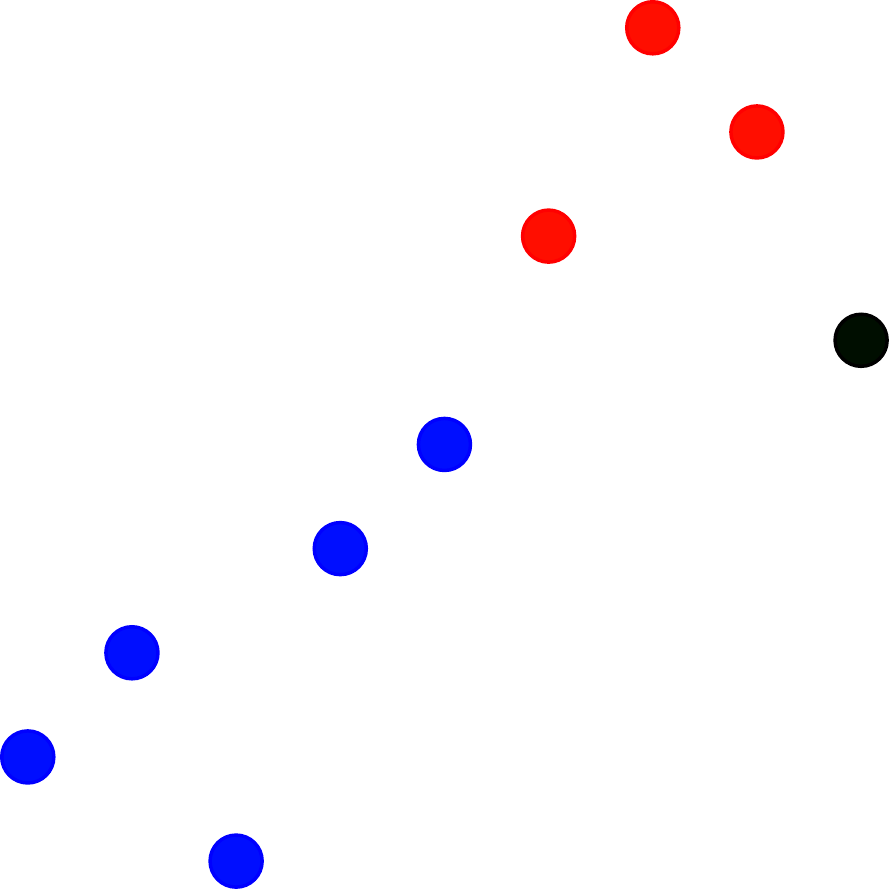}\end{array}.\] 
Then we have \[u=23145=\begin{array}{l}\includegraphics[height=1.324cm]{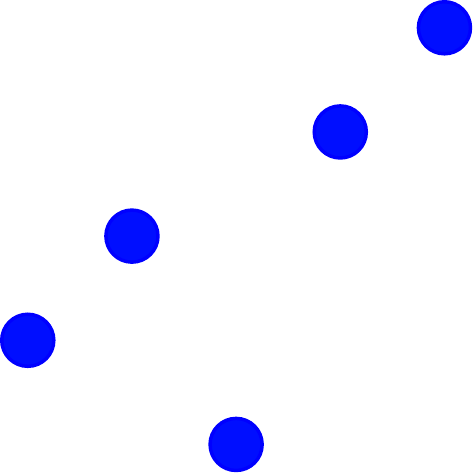}\end{array}\quad \text{and}\quad v=132=\begin{array}{l}\includegraphics[height=0.74cm]{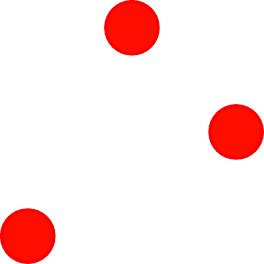}\end{array}.\] We have $m=5$ and $r=x^{-1}(6)=6$. The sequence \[x(r-(n-m)+1)\cdots x(r-1)=x(3)x(4)x(5)=798\] has standardization $z=132$. We must choose a permutation \[z'\in\MM(z)\cap\Eeta(\Tam_{3})=\MM(132)\cap\Eeta(\Tam_3);\] in this particular example, our only choice is to set $z'=123$. The permutation obtained from $x$ by deleting the entries $7,8,9$ is \[w=236541=\begin{array}{l}\includegraphics[height=1.616cm]{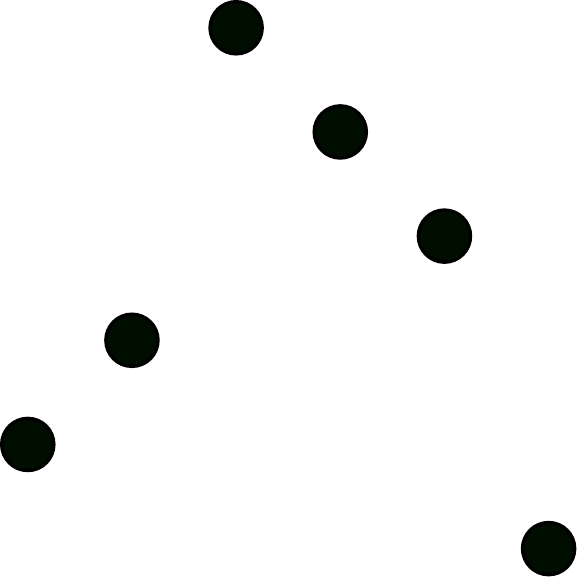}\end{array}.\] As observed in the proof of \Cref{lem:ribute}, we have $u\oplus 1=231456\in\MM(w)$. Finally, we set \[\widetilde y=u\oplus (z'\ominus 1)=231457896=\begin{array}{l}\includegraphics[height=2.493cm]{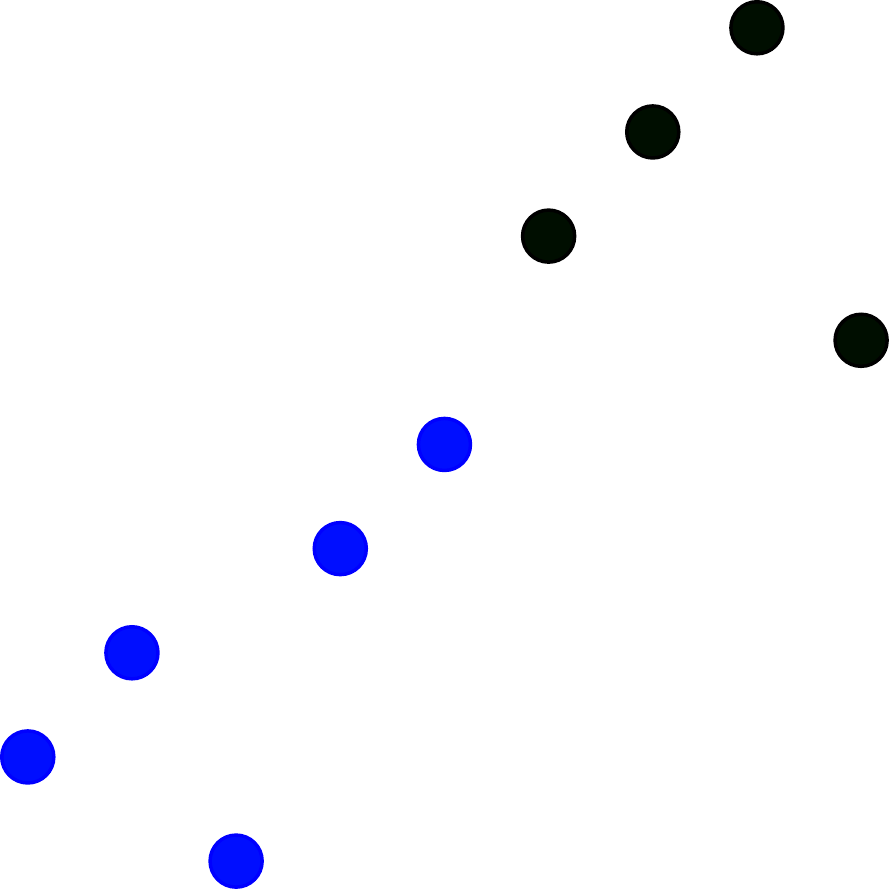}\end{array},\] and we observe that \[\widetilde y\ominus 1=(u\oplus (z'\ominus 1))\ominus 1=3\,4\,2\,5\,6\,8\,9\,10\,7\,1=\begin{array}{l}\includegraphics[height=2.784cm]{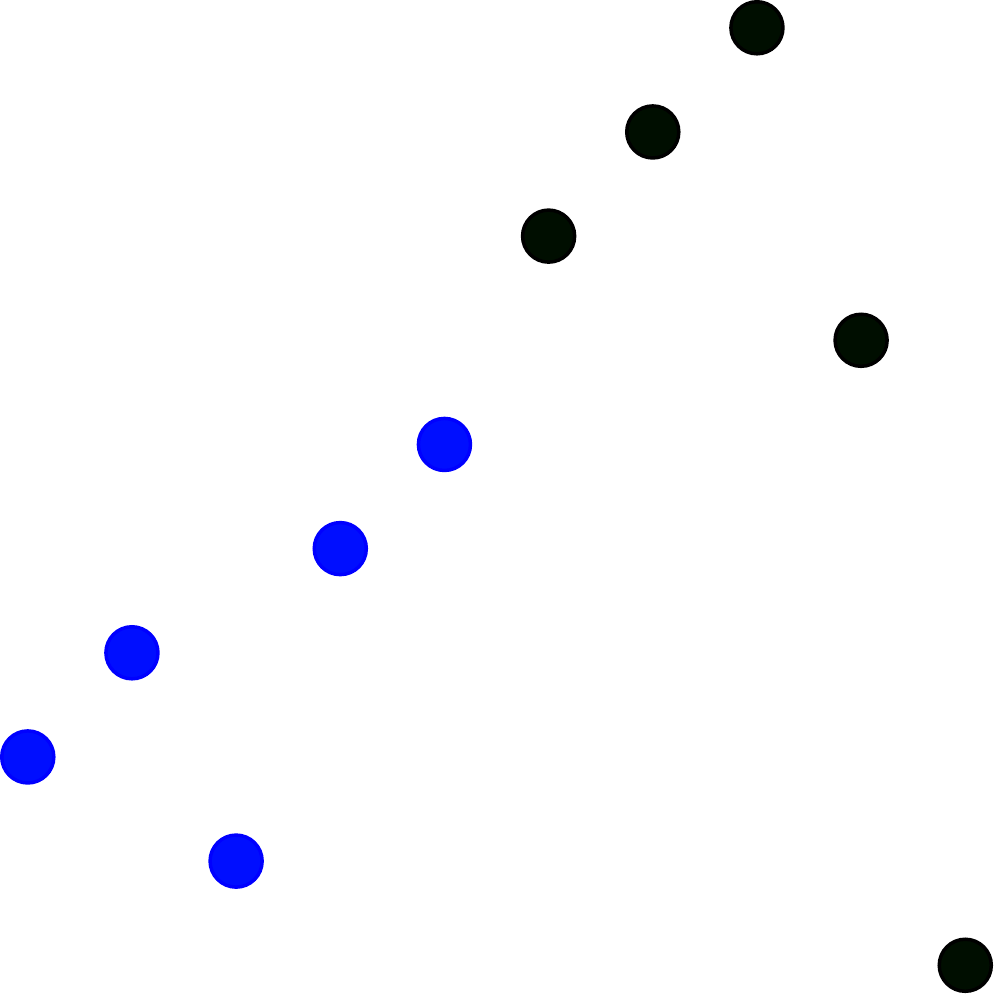}\end{array}\] is indeed even-districted. 
\end{example}

\begin{lemma}\label{lem:ribute_implies_Eeta}
Suppose $z\in\Tam_n$ is even-districted and $z'$ is a decomposable element of $\MM(z)$. Then the first component of $z'$ is not even-districted. 
\end{lemma}
\begin{proof}
Let us write $z=(y_1\oplus\cdots\oplus y_r)\ominus 1$, where $y_1,\ldots,y_r$ are indecomposable. Let $y_r=\widehat y\ominus 1$. The assumption that $z$ is even-districted is equivalent to the assertion that $r$ is even and $y_1\oplus\cdots\oplus y_{r-1}\oplus \widehat y$ is an Eeta win. In particular, it follows from \Cref{prop:Tamari_product} that $y_1,\ldots,y_{r-1}$ are Eeta wins. According to our description of Tamari lattice Ungar moves, we can write $z'=((y_1'\oplus\cdots\oplus y_{r-1}')\ominus 1)\oplus y_r'$, where $y_i'\in\MM(y_i)$ for all $1\leq i\leq r$. The first component of $z'$ is $(y_1'\oplus\cdots\oplus y_{r-1}')\ominus 1$, so we need to show that this is not even-districted. We consider a few cases. 

\medskip 

\noindent {\bf Case 1.} Suppose that $y_j'\neq y_j$ for some $j\in[r-2]$. Since $y_j$ is an Eeta win, $y_j'$ is an Atniss win. This implies (by \Cref{prop:Tamari_product}) that the standardization of the permutation obtained by deleting the last two entries from $(y_1'\oplus\cdots\oplus y_{r-1}')\ominus 1$ is an Atniss win, so $(y_1'\oplus\cdots\oplus y_{r-1}')\ominus 1$ is not even-districted. 

\medskip 

\noindent {\bf Case 2.} Suppose that $y_j'=y_j$ for all $j\in[r-2]$ and that $y_{r-1}'$ is indecomposable. Then $y_1'\oplus\cdots\oplus y_{r-1}'$ has $r-1$ components, so $(y_1'\oplus\cdots\oplus y_{r-1}')\ominus 1$ is not even-districted because $r-1$ is odd. 

\medskip 

\noindent {\bf Case 3.} Suppose that $y_j'=y_j$ for all $j\in[r-2]$ and that $y_{r-1}'$ is decomposable. Let us write $y_{r-1}=(x_1\oplus \cdots\oplus x_t)\ominus 1$ for some indecomposable permutations $x_1,\ldots,x_t$. According to our description of Tamari lattice Ungar moves, we must have $y_{r-1}'=((x_1'\oplus\cdots\oplus x_{t-1}')\ominus 1)\oplus x_t'$, where $x_i'\in\MM(x_i)$ for all $1\leq i\leq t$. Then we have \[z=\begin{array}{l}\includegraphics[height=5cm]{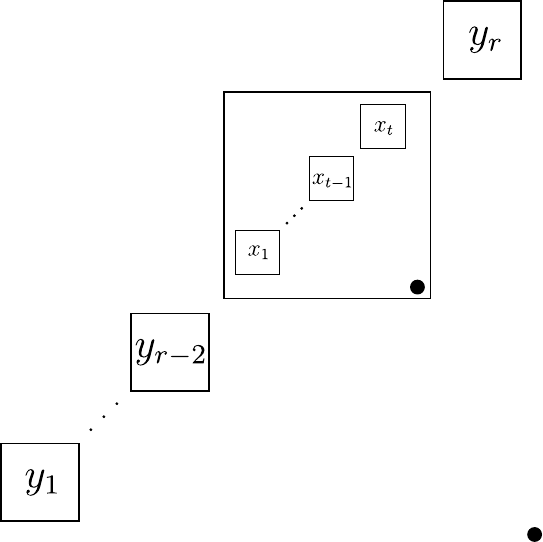}\end{array}\qquad\qquad\text{and}\qquad\qquad z'=\begin{array}{l}\includegraphics[height=5cm]{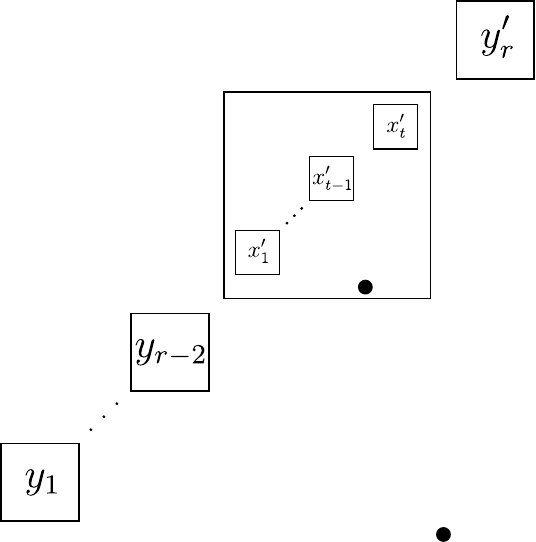}\end{array}.\] \Cref{lem:toP2} tells us that there exists an Eeta win $x_t''$ in $\MM(x_t)$. Then $((x_1'\oplus\cdots\oplus x_{t-1}')\ominus 1)\oplus x_t''$ is in $\MM(y_{r-1})$ and is not equal to $y_{r-1}$ because $y_{r-1}$ is indecomposable. Because $y_{r-1}$ is an Eeta win, this implies that $((x_1'\oplus\cdots\oplus x_{t-1}')\ominus 1)\oplus x_t''$ is an Atniss win. But $x_t''$ is an Eeta win, so it follows from \Cref{prop:Tamari_product} that $(x_1'\oplus\cdots\oplus x_{t-1}')\ominus 1$ is an Atniss win. This shows that some non-final component of $y_{r-1}'$ is an Atniss win, so some non-final component of $y_1'\oplus\cdots\oplus y_{r-1}'$ is an Atniss win. By \Cref{prop:Tamari_product}, the standardization of the permutation obtained by deleting the last two entries from $(y_1'\oplus\cdots\oplus y_{r-1}')\ominus 1$ is an Atniss win, so $(y_1'\oplus\cdots\oplus y_{r-1}')\ominus 1$ is not even-districted.  
\end{proof}

We can now prove \Cref{prop:Tamari_characterization}. 

\begin{proof}[Proof of \Cref{prop:Tamari_characterization}]
It is easy to check that the desired result holds when $n\leq 2$. Therefore, we may assume $n\geq 3$ and proceed by induction on $n$. Let $w\in\Tam_n$ be indecomposable. We will prove that $w\in\Eeta(\Tam_n)$ if and only if $w$ is even-districted. We may also apply induction on the lattice $\Tam_n$. In other words, we may assume that the set of indecomposable Eeta wins that are less than $w$ in $\Tam_n$ is equal to the set of even-districted permutations that are less than $w$ in $\Tam_n$. 

Assume first that $w$ is even-districted. Suppose $x\in\MM(w)\setminus\{w\}$; we need to show that $x$ is an Atniss win. If $x$ is decomposable, then we can set $z=w$ and $z'=x$ in \Cref{lem:ribute_implies_Eeta} to find that the first component of $x$ is not even-districted. By induction, this implies that the first component of $x$ is an Atniss win, so it follows from \Cref{prop:Tamari_product} that $x$ is an Atniss win. 

Now assume $x$ is indecomposable. Let $x=x'\ominus 1$. Because $x<w$ in $\Tam_n$, we can use induction to see that $x$ is an Atniss win if and only if it is not even-districted; thus, we need to show that $x$ is not even-districted. Let $q$ be the standardization of the permutation obtained by deleting the last two entries from $x$. It suffices to show either that $x'$ has an odd number of components or that $q$ is an Atniss win. Let us write $w=(u_1\oplus \cdots\oplus u_r)\ominus 1$, where $u_1,\ldots,u_r$ are indecomposable. Then $x'=u_1'\oplus\cdots\oplus u_r'$, where $u_i'\in\MM(u_i)$ for all $1\leq i\leq r$. Let $u_r=y\ominus 1$. Our assumption that $w$ is even-districted tells us that $r$ is even and that $u_1\oplus\cdots\oplus u_{r-1}\oplus y$ is an Eeta win. It follows from \Cref{prop:Tamari_product} that $u_1,\ldots,u_{r-1},y$ are Eeta wins. 
We now consider three cases. 

\medskip 

\noindent {\bf Case 1.} Suppose $u_j'\neq u_j$ for some $j\in[r-1]$. Because $u_j$ is an Eeta win and $u_j'\in\MM(u_j)$, we know that $u_j'$ is an Atniss win. It follows from \Cref{prop:Tamari_product} that $q$ is an Atniss win, so $x$ is not even-districted. 

\medskip 

\noindent {\bf Case 2.} Suppose that $u_j'=u_j$ for all $j\in[r-1]$ and that $u_r'$ is indecomposable. Then $u_r'=y'\ominus 1$ for some $y'\in\MM(y)\setminus\{y\}$. Since $y$ is an Eeta win, $y'$ is an Atniss win. Thus, $q=u_1\oplus\cdots\oplus u_{r-1}\oplus y'$ is an Atniss win by \Cref{prop:Tamari_product}. This proves that $x$ is not even-districted. 

\medskip 

\noindent {\bf Case 3.} Suppose that $u_j'=u_j$ for all $j\in[r-1]$ and that $u_r'$ is decomposable. We can write $y=v_1\oplus\cdots\oplus v_t$, where $v_1,\ldots,v_t$ are the components of $y$. Because $y$ is an Eeta win, we know by \Cref{prop:Tamari_product} that $v_1,\ldots,v_t$ are Eeta wins. Our induction hypothesis guarantees that $v_1,\ldots,v_t$ are even-districted. Since $u_r'$ is decomposable, we have $u_r'=((v_1'\oplus\cdots\oplus v_{t-1}')\ominus 1)\oplus v_t'$, where $v_i'\in\MM(v_i)$ for all $1\leq i\leq t$. If $v_t'$ is indecomposable, then $u_r'$ has exactly $2$ components, so $x'=u_1\oplus\cdots\oplus u_{r-1}\oplus u_r'$ has exactly $r+1$ components. In this case, $x$ is not even-districted because $r+1$ is odd. Thus, we may assume that $v_t'$ is decomposable. Applying \Cref{lem:ribute_implies_Eeta} with $z=v_t$ and $z'=v_t'$, we find that the first component of $v_t'$ is not even-districted. By induction, the first component of $v_t'$ is an Atniss win. The first component of $v_t'$ is a non-final component of $u_r'$, so it is also a non-final component of $x'$. This implies that the first component of $v_t'$ is also a component of $q$, so $q$ is an Atniss win by \Cref{prop:Tamari_product}. 

\medskip 

We have proven that if $w$ is even-districted, then it is an Eeta win. To prove the converse, let us now assume $w$ is not even-districted; our goal is to show that $w$ is an Atniss win. Hence, we need to show that there exists an Eeta win in $\MM(w)\setminus\{w\}$. Let us write $w=w'\ominus 1$, and let $v$ be the final component of $w'$. Let $v=v'\ominus 1$. By \Cref{lem:toP2}, there exist Eeta wins $z\in\MM(v)$ and $z'\in\MM(v')$. If $w'$ is indecomposable, then $v=w'$, so $1\oplus z$ is an Eeta win in $\MM(w)\setminus\{w\}$. Thus, we may assume $w'$ is decomposable and write $w'=u\oplus v$ for some (possibly decomposable) permutation $u$. We consider three cases. 

\medskip 

\noindent {\bf Case 1.} Suppose $u$ is an Atniss win. Then there exists an Eeta win $\widehat u\in\MM(u)\setminus\{u\}$. Note that $(\widehat u\oplus(z'\ominus 1))\ominus 1\in\MM(w)\setminus\{w\}$. If $\widehat u$ has an odd number of components, then we can use \Cref{prop:Tamari_product} to see that $(\widehat u\oplus(z'\ominus 1))\ominus 1$ is even-districted (because $\widehat u$ and $z'$ are Eeta wins), so it follows by induction that $(\widehat u\oplus(z'\ominus 1))\ominus 1$ is an Eeta win. Now suppose $\widehat u$ has an even number of components. According to \Cref{lem:ribute}, there exists $\widehat u'\in\MM(u)$ such that $\widehat u'\ominus 1$ is even-districted. By induction, $\widehat u'\ominus 1$ is an Eeta win. Consequently, $(\widehat u'\ominus 1)\oplus z$ is an Eeta win in $\MM(w)\setminus\{w\}$. 

\medskip 

\noindent {\bf Case 2.} Suppose $u$ is an Eeta win with an even number of components. Since $u\in\MM(u)$, we can appeal to \Cref{lem:ribute} to find that there exists $u'\in\MM(u)$ such that $u'\ominus 1$ is even-districted. By induction, $u'\ominus 1$ is an Eeta win. Consequently, $(u'\ominus 1)\oplus z$ is an Eeta win in $\MM(w)\setminus\{w\}$. 

\medskip 

\noindent {\bf Case 3.} Suppose $u$ is an Eeta win with an odd number of components. Then $u\oplus z'$ is an Eeta win by \Cref{prop:Tamari_product}, so $(u\oplus (z'\ominus 1))\ominus 1$ is even-districted. Also, $(u\oplus (z'\ominus 1))\ominus 1$ is in $\MM(w)\setminus\{w\}$ (notice that $(u\oplus (z'\ominus 1))\ominus 1 \neq w$ by our assumption that $w$ is not even-districted). This implies that $(u\oplus (z'\ominus 1))\ominus 1<w$ in $\Tam_n$, so by induction, $(u\oplus (z'\ominus 1))\ominus 1$ is an Eeta win. 
\end{proof}

Having recursively characterized Eeta wins in Tamari lattices via \Cref{prop:Tamari_product,prop:Tamari_characterization}, we can now enumerate them. 

\begin{proof}[Proof of \Cref{thm:Tamari}]
Let $G(z)=\sum_{n\geq 1}g_nz^n$, where $g_n$ is the number of even-districted elements of $\Tam_n$. For $n\geq 3$, it follows from \Cref{prop:Tamari_product,prop:Tamari_characterization} that every even-districted element of $\Tam_n$ can be written uniquely in the form $(u_1\oplus\cdots\oplus u_k\oplus((u_{k+1}\oplus\cdots\oplus u_{r})\ominus 1))\ominus 1$, where $k$ is odd, $r\geq k$, and $u_1,\ldots,u_r$ are even-districted. Thus, \[g_n=\sum_{\substack{r\geq k\geq 1 \\ k\text{ odd}}}\sum_{\substack{n_1,\ldots,n_r\geq 1 \\ n_1+\cdots+n_r=n-2}}g_{n_1}\cdots g_{n_r}=\sum_{r\geq 1}\left\lceil r/2\right\rceil\sum_{\substack{n_1,\ldots,n_r\geq 1 \\ n_1+\cdots+n_r=n-2}}g_{n_1}\cdots g_{n_r}.\]
Translating this recurrence into generating functions yields 
\begin{align}\label{eq:G(z)}
G(z)&=z+z^2\sum_{r\geq 1}\left\lceil r/2\right\rceil G(z)^r \nonumber \\ 
&=z+z^2\sum_{m\geq 1}m(G(z)^{2m-1}+G(z)^{2m}) \nonumber \\ 
&=z+z^2(G(z)+G(z)^2)\sum_{m\geq 1}m(G(z)^2)^{m-1} \nonumber \\ 
&=z+z^2\frac{G(z)+G(z)^2}{(1-G(z)^2)^2}.
\end{align}

Let $F(z)=\sum_{n\geq 1}|\Eeta(\Tam_n)|z^n$. According to \Cref{prop:Tamari_characterization}, $G(z)$ is the generating function for indecomposable Tamari lattice Eeta wins. As a consequence, $F(z)=\frac{G(z)}{1-G(z)}$. Equivalently, $G(z)=\frac{F(z)}{1+F(z)}$. After substituting this into \eqref{eq:G(z)} and performing basic algebraic manipulations, we find that $Q(F(z),z)=0$, where \[Q(y,z)=z + (-1 + 3 z + z^2) y + (-2 + 2 z + 3 z^2) y^2 + 3 z^2 y^3 + z^2 y^4.\] 

The method used to derive the asymptotics in the statement of the theorem is routine and is discussed in \cite[Chapter~VII]{Flajolet}; we will just sketch the details. Let $\rho=\lim\limits_{n\to\infty}|\Eeta(\Tam_n)|^{1/n}$. The discriminant of $Q(y,z)$ with respect to $y$ is $z^2\widehat Q(z)$, where \[\widehat Q(z)=32 - 32 z - 155 z^2 - 20 z^3 - 148 z^4 + 60 z^5 - 8 z^6 - 4 z^7.\] Pringsheim's theorem states that $1/\rho$ must be a positive real root of this discriminant, so $\rho$ is a positive real root of $z^7\widehat Q(1/z)$. One can check that $z^7\widehat Q(1/z)$ has a unique positive real root. One can then use a computer algebra software such as Maple to expand $F(z)$ as a Puiseux series centered at $1/\rho$; the result is $\beta_0+{\beta_1(z-1/\rho)^{1/2}}+o((z-\rho)^{1/2})$ for some explicitly computable algebraic numbers $\beta_0$ and $\beta_1$. Following the discussion in \cite[Chapter~VII]{Flajolet}, this expansion transfers into an asymptotic formula of the form \[|\Eeta(\Tam_n)|\sim \frac{\gamma}{\sqrt{\pi}}n^{-3/2}\rho^n,\] and one can use a computer algebra software to find that the minimal polynomial of $\gamma$ is as stated in the theorem.
\end{proof}

\section{Open Problems}\label{sec:open}

\subsection{The Weak Order}\label{sec:open_weak}
Although \Cref{thm:Weak} provides an asymptotic upper bound for the number of Eeta wins in the weak order on $S_n$, we are still far from fully understanding these Ungar games. It would be interesting to improve the upper bound in \Cref{thm:Weak} or find a nontrivial lower bound.  For instance, does the number of Eeta wins grow more like $c^{n} n!$ or more like $(n!)^c$ (each for some $c<1$)?

Consider the set $B$ of permutations from the statement of \Cref{lem:pattern_avoidance}. We deduced \Cref{thm:Weak} from that lemma and a known asymptotic estimate for the number of permutations in $S_n$ that consecutively avoid $1324$. It could be interesting to more accurately enumerate (either exactly or asymptotically) the permutations that consecutively avoid \emph{all} of the patterns in $B$; this would immediately yield an improvement upon \Cref{thm:Weak}. 

An earlier version of this work included the conjecture that if a permutation $w$ is an Eeta win in the weak order on $S_n$, then $w$ has at most $\frac{n-1}{2}$ descents.  This conjecture was disproved by Evan Bailey, who used a computer to find all counterexamples for $n \leq 14$.  The first counterexamples appear for $n=10$; for example, the permutation with one-line notation $3, 10, 9, 8, 4, 7, 2, 5, 1, 6$ is an Eeta win with $5$ descents.

\subsection{Other Lattices}

\Cref{thm:Young} considers a large class of intervals in Young's lattice and characterizes which of them are Eeta wins. It would be interesting to extend this characterization to \emph{all} intervals in Young's lattice. 

Of course, it would also be interesting to study Ungar games on other lattices beyond those considered here. For example, since Young's lattice is $J(\mathbb N^2)$, it is natural to ask what can be said about Ungar games on principal order ideals of $J(\mathbb N^3)$. Another well-studied lattice that is similar in many ways to Young's lattice is the Young--Fibonacci lattice, which was introduced by Fomin~\cite{Fomin} and Stanley \cite{StanleyDifferential}; note, however, that this lattice is not distributive. The number of elements of rank $n$ in the Young--Fibonacci lattice is the Fibonacci number $f_n$, where we use the conventions $f_0=f_1=1$ and $f_{n}=f_{n-1}+f_{n-2}$ for $n\geq 2$. 

\begin{conjecture}
    For $n\geq 2$, the number of Eeta wins of rank $n$ in the Young--Fibonacci lattice is $f_{n-2}+(-1)^{n}$.
\end{conjecture}

\begin{figure}[ht]
\[\hspace{-2.2cm}\begin{array}{l}\includegraphics[height=4.5cm]{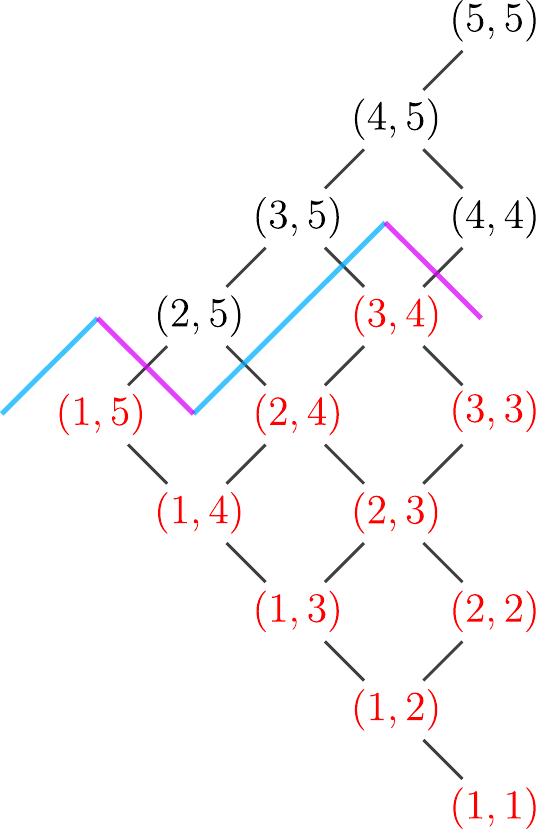}\end{array}\quad\longleftrightarrow \quad{\color{SkyBlue}1}{\color{MyPurple}0}{\color{SkyBlue}11}{\color{MyPurple}0}\]
\caption{An order ideal of the shifted staircase $\SSS_5$ is shown in {\color{red}red}. This order ideal is uniquely determined by a path of {\color{SkyBlue}up} and {\color{MyPurple}down} steps lying just above it, and that path corresponds to the length-$5$ binary string ${\color{SkyBlue}1}{\color{MyPurple}0}{\color{SkyBlue}11}{\color{MyPurple}0}$.}\label{fig:shifted_staricase}
\end{figure}

The $n$-th \dfn{shifted staircase} is the subposet $\SSS_n$ of $\mathbb N^2$ consisting of all pairs $(i,j)$ such that ${1\leq i\leq j\leq n}$. There is a natural bijection between order ideals of $\SSS_n$ and binary strings of length $n$; we illustrate this bijection for $n=5$ in \Cref{fig:shifted_staricase}. A \dfn{$0$-block} (respectively, \dfn{$1$-block}) in a binary string is a maximal consecutive substring of $0$'s (respectively, $1$'s). Note that $J(\SSS_n)$ is generally not isomorphic to an interval in Young's lattice, so we cannot apply \Cref{thm:Young} to understand its Atniss wins and Eeta wins. Nevertheless, the following characterization seems to hold.

\begin{conjecture}
An order ideal of $\SSS_n$ is an Eeta win in $J(\SSS_n)$ if and only if its corresponding length-$n$ binary string ends with $0$ and does not contain an odd-length $0$-block immediately followed by an odd-length $1$-block. 
\end{conjecture}

\subsection{Complexity}
It is natural to consider Ungar games from the point of view of complexity theory.  A \dfn{boolean formula} is an expression on $n$ boolean inputs using the usual binary operations $\lorb$ and $\landb$ and the unary operation $\negb$.  We are interested in the class of boolean formulas whose truth value can be computed with a circuit of depth $O(\log(n))$~\cite{Barrington}.  A decision problem is called \dfn{$\mathsf{NC}^1$-hard} if any such formula is linearly reducible to it. 

In~\cite{kalinich2012flipping}, Kalinich showed that poset games are $\mathsf{NC}^1$-hard.  We can adapt this argument to Ungar games.

\begin{theorem}
Ungar games are $\mathsf{NC}^1$-hard.
\end{theorem}

\begin{proof}
As in~\cite{kalinich2012flipping}, we show that we can construct Ungar games that encode the boolean formula value problem with only linear blowup---that is, we produce a lattice that is an Eeta win if and only if the formula evaluates to 1 using the given inputs. 

We represent posets as Hasse diagrams so that the data for a poset is polynomial in the number of its elements.  Noting that the lattice with $1$ element is an Eeta win and the lattice with $2$ elements is an Atniss win, we see that it suffices to construct the $\lorb$ of two games and the 
$\negb$ of a game.  This is carried out in~\Cref{fig:vee_and_not}, at the expense of $7$ extra elements per $\lorb$ and $1$ extra element per $\negb$.  

By our assumption that the depth of the given formula is $O(\log(n))$, the resulting poset has $n^{O(1)}$ elements. It is straightforward to see by induction that this poset is indeed a lattice. 
\end{proof}
    
\begin{figure}[htbp]
        \raisebox{-.5\height}{\scalebox{1}{\begin{tikzpicture}[scale=1]
        \node[shape=circle,fill=black, scale=0.5] (0) at (0,0) {};
        \node[shape=circle,fill=black, scale=0.5] (1) at (-1,1) {};
        \node[shape=circle,fill=black, scale=0.5] (2) at (1,1) {};
        \node[shape=circle,draw] (3) at (-2,2) {$x$};
        \node[shape=circle,draw] (4) at (2,2) {$y$};
        \node[shape=circle,fill=black, scale=0.5] (5) at (0,2) {};
        \node[shape=circle,fill=black, scale=0.5] (6) at (-1,3) {};
        \node[shape=circle,fill=black, scale=0.5] (7) at (1,3) {};
        \node[shape=circle,fill=black, scale=0.5] (8) at (0,4) {};
        \draw[very thick] (0) to (1) to (3);
        \draw[very thick] (3) to (6) to (8);
        \draw[very thick] (0) to (2) to (4);
        \draw[very thick] (4) to (7) to (8);
        \draw[very thick] (0) to (5) to (6);
        \draw[very thick] (5) to (7);
        \end{tikzpicture}}} \hspace{3em} \raisebox{-.5\height}{\scalebox{1}{\begin{tikzpicture}[scale=1]
        \node[shape=circle,fill=black, scale=0.5] (1) at (0,1) {};
        \node[shape=circle,draw] (0) at (0,0) {$x$};
        \draw[very thick] (0) to (1);
        \end{tikzpicture}}}\hspace{3em}
        \raisebox{-.5\height}{\scalebox{.5}{\begin{tikzpicture}[scale=1]
        \node[shape=circle,fill=black, scale=0.5] (0) at (0,0) {};
        \node[shape=circle,fill=black, scale=0.5] (1) at (-1,1) {};
        \node[shape=circle,fill=black, scale=0.5] (2) at (-2,2) {};
        \node[shape=circle,fill=black, scale=0.5] (3) at (-3,3) {};
        \node[shape=circle,fill=black, scale=0.5] (4) at (-4,4) {};
        \node[shape=circle,fill=black, scale=0.5] (5) at (-4.5,5) {};
        \node[shape=circle,draw] (6) at (-5,6) {$x_1$};
        \node[shape=circle,fill=black, scale=0.5] (7) at (-4.5,7) {};
        \node[shape=circle,fill=black, scale=0.5] (8) at (-4,8) {};
        \node[shape=circle,fill=black, scale=0.5] (9) at (-3,9) {};
        \node[shape=circle,fill=black, scale=0.5] (10) at (-2,10) {};
        \node[shape=circle,fill=black, scale=0.5] (11) at (-2,11) {};
        \node[shape=circle,fill=black, scale=0.5] (12) at (-1,12) {};
        \node[shape=circle,fill=black, scale=0.5] (13) at (0,13) {};
        \node[shape=circle,fill=black, scale=0.5] (14) at (0,14) {};
        \node[shape=circle,fill=black, scale=0.5] (15) at (-1,9) {};
        \node[shape=circle,fill=black, scale=0.5] (16) at (-2,8) {};
        \node[shape=circle,draw] (17) at (-1,6) {$x_3$};
        \node[shape=circle,fill=black, scale=0.5] (17a) at (-1,7) {};
        \node[shape=circle,fill=black, scale=0.5] (18) at (-1,4) {};
        \node[shape=circle,fill=black, scale=0.5] (19) at (-3.5,7) {};
        \node[shape=circle,draw] (20) at (-3,6) {$x_2$};
        \node[shape=circle,fill=black, scale=0.5] (21) at (-4,6) {};
        \node[shape=circle,fill=black, scale=0.5] (22) at (3,9) {};
        \node[shape=circle,fill=black, scale=0.5] (23) at (3,8) {};
        \node[shape=circle,fill=black, scale=0.5] (24) at (3.5,7) {};
        \node[shape=circle,fill=black, scale=0.5] (25) at (2.5,7) {};
        \node[shape=circle,fill=black, scale=0.5] (26) at (3,6) {};
        \node[shape=circle,draw] (27) at (4,6) {$x_5$};
        \node[shape=circle,draw] (28) at (2,6) {$x_4$};
        \node[shape=circle,fill=black, scale=0.5] (29) at (2.5,5) {};
        \node[shape=circle,fill=black, scale=0.5] (30) at (3.5,5) {};
        \node[shape=circle,fill=black, scale=0.5] (31) at (3,4) {};
        \node[shape=circle,fill=black, scale=0.5] (32) at (1.5,2) {};
        \node[shape=circle,fill=black, scale=0.5] (33) at (-3.5,5) {};
        \node[shape=circle,fill=black, scale=0.5] (34) at (1,12) {};
        \node[shape=circle,fill=black, scale=0.5] (35) at (0,11) {};
        \draw[very thick] (0) to (1) to (2) to (3) to (4) to (5) to (6) to (7) to (8) to (9) to (10) to (11) to (12) to (13) to (14);
        \draw[very thick] (2) to (18) to (17) to (17a) to (15) to (10);
        \draw[very thick] (2) to (16) to (15);
        \draw[very thick] (16) to (9);
        \draw[very thick] (4) to (33) to (20) to (19) to (8);
        \draw[very thick] (4) to (21) to (19);
        \draw[very thick] (21) to (7);
        \draw[very thick] (0) to (35) to (34);
        \draw[very thick] (35) to (12);
        \draw[very thick] (0) to (32) to (31) to (30) to (27) to (24) to (23) to (22) to (34) to (13);
        \draw[very thick] (31) to (29) to (28) to (25) to (23);
        \draw[very thick] (31) to (26) to (25);
        \draw[very thick] (26) to (24);
        \end{tikzpicture}}}

    \caption{Left: a lattice encoding the boolean formula $x\, \lorb \, y$; middle: a lattice encoding $\negb \, x$; right: a lattice encoding $(x_1  \, \lorb  \, x_2  \, \lorb  \, (\negb \, x_3))  \, \landb  \, (x_4  \, \lorb  \, x_5)$.  Each variable should be replaced by the $1$-element lattice (corresponding to setting the variable to $1$) or the $2$-element lattice (corresponding to setting the variable to $0$).}
    \label{fig:vee_and_not}
    \end{figure}
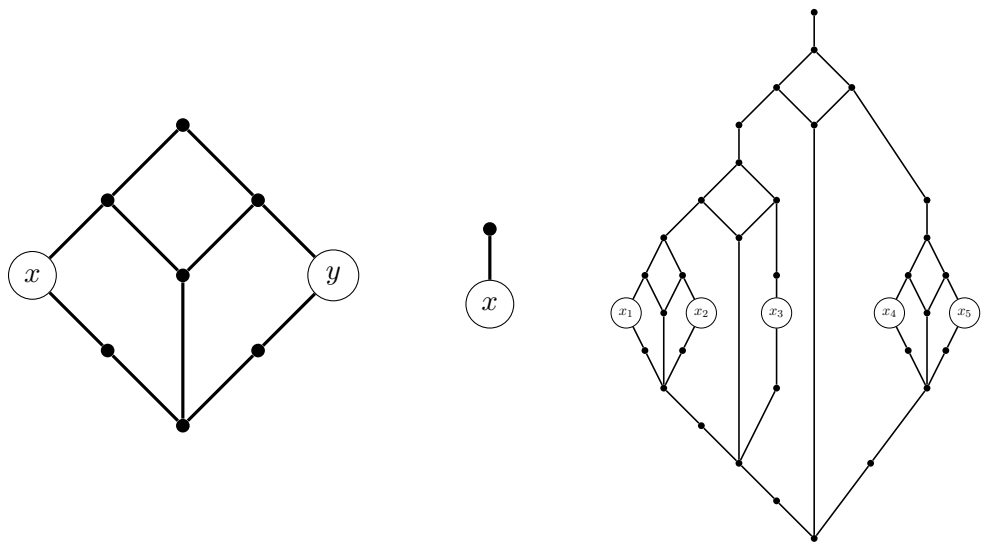

By analogy with poset games, it is reasonable to consider Ungar games on distributive lattices $J(P)$, so that the game can be played on $P$ itself.  Building on work of Schaeffer~\cite{schaefer1978complexity}, Grier proved that poset games are PSPACE-complete~\cite{grier2013deciding}.  It is not so easy to adapt Grier's argument from poset games to Ungar games on distributive lattices.

\begin{question}\label{q:pspace}
Are Ungar games on distributive lattices $J(P)$ $\PSPACE$-complete in $|P|$? 
\end{question}

\section*{Acknowledgements}
Colin Defant was supported by the National Science Foundation under Award No.\ 2201907 and by a Benjamin Peirce Fellowship at Harvard University.  Noah Kravitz was supported in part by an NSF Graduate Research Fellowship  (grant DGE--2039656).  Nathan Williams was partially supported by the National Science Foundation under Award No.\ 2246877.  We are grateful to Evan Bailey for disproving the conjecture discussed in \Cref{sec:open_weak} and providing useful feedback on \Cref{q:pspace}; we thank Jay Pantone for helpful conversations. We also thank the anonymous referees for providing helpful suggestions.

\end{document}